\documentclass[11pt]{amsart}
\usepackage[a4paper, top=2.8cm, left=2cm, right=2cm, bottom=2cm]{geometry}
\usepackage{graphicx}
\usepackage{amsfonts}
\usepackage{amsmath}
\usepackage{amssymb}
\usepackage{amsthm}
\usepackage{newlfont}
\usepackage{yhmath}
\usepackage{color}

\newtheorem{thm}{Theorem}[section]
\newtheorem{lemma}{Lemma}[section]
\newtheorem{rmk}{Remark}[section]
\newtheorem{cor}{Corollary}[section]

\newcommand{\p}{\mathbb{P}}
\newcommand{\e}{\mathbb{E}}
\newcommand{\bi}{\mathbf{i}}

\begin{document}
\title{Some Moderate Deviations for Ewens-Pitman Sampling Model}
\author{Youzhou Zhou}
\address{Department of Mathematical Science\\
Xi'an Jiaotong-Liverpool University\\
111 Ren'ai Road,
Suzhou Dushu Lake Science and Education Innovation District,
Suzhou Industrial Park,
Suzhou, 
P. R. China
215123}
\email{youzhou.zhou@xjtlu.edu.cn}
\date{\today}
\thanks{}
\begin{abstract}
Ewens-Pitman model has been successfully applied to various fields including Bayesian statistics. There are four important estimators $K_{n},M_{l,n}$,$K_{m}^{(n)},M_{l,m}^{(n)}$. In particular, $M_{1,n}, M_{1,m}^{(n)}$ are related to discovery probability. Their asymptotic behavior, such as large deviation principle, has already been discussed in \cite{MR1661315},\cite{MR3167885} and \cite{MR3335831}. Moderate deviation principle is also discussed in \cite{MR3850070} with some speed restriction. In this article, we will apply complex asymptotic analysis to show that this speed restriction is unnecessary. 
\end{abstract}
\maketitle
\allowdisplaybreaks
\section{Introduction}
Ewens-Pitman sampling model proposed by \cite{MR1337249} generates an exchangeable random partition \cite{MR509954}. Due to Kingman's correspondence theorem (or de Finetti theorem), each exchangeable random partition structure associate a random probability, which serves as a prior distribution in nonparametric Bayesian statistics.
Therefore, Ewens-Pitman sampling model has found its applications in various fields including machine learning theory. To define this model, let a polish space $\mathbb{X}$ be type space, and a diffuse probability $\nu$ be type distribution. For $\alpha\in(0,1),\theta>-\alpha$, we define a sampling sequence $(X_{i})_{i\geq1}$ as follows:
 \begin{align*}
 \p(X_{1}\in\cdot)=&\nu(\cdot)\\
 \p(X_{n+1}\in\cdot\mid X_{1},\cdots,X_{n})=&\frac{\theta+j\alpha}{\theta+n}\nu(\cdot)+\frac{1}{\theta+n}\sum_{l=1}^{j}(n_{l}-\alpha)\delta_{X_{l}^{*}}(\cdot)
 \end{align*}
 where $X_{1}^{*},\cdots,X_{l}^{*}$ are the distinctive values of finite sample $X_{1},\cdots,X_{n}$, and $N_{n}=(n_{1},\cdots,n_{l})$ are the type frequencies. For a given sample $X_{1},\cdots,X_{n}$, we define $K_{n}$ to be the number of different types, and $M_{l,n}$
 to be the number of types with frequency $l$. If we consider an extra sample $X_{n+1},\cdots,X_{n+m}$, we can similarly define $K_{m}^{(n)}$ to be the total number of brand new types(no appearance in $X_{1},\cdots,X_{n}$) and $M_{l,m}^{(n)}$ to be the number of brand new types with frequency $l$ in the extra sample.  In \cite{MR2434179} Ewens-Pitman model, or more generally Gibbs-type random partition model \cite{MR2160320},
 was applied to the analysis of expressed sequence tags. Conditioning on sample $X_{1},\cdots,X_{n}$, one can show that $K_{m}^{(n)}$ and  $M_{l,m}^{(n)}$ are posterior estimators. In particular, $M_{1,m}^{(n)}$ corresponds to discovery probability like Turing-Good estimator. Note that $K_{m}^{(n)}$ and  $M_{l,m}^{(n)}$ do not count the types appear in sample $X_{1},\cdots,X_{n}$. If we do consider those types, then we have 
 $\tilde{K}_{m}^{(n)}$ and $\tilde{M}_{l,m}^{(n)}$. However, $0\leq \tilde{K}_{m}^{(n)}-K_{m}^{(n)}\leq n$ and 
 $0\leq \tilde{M}_{l,m}^{(n)}-M_{l,m}^{(n)}\leq n$, so $K_{m}^{(n)}$ and  $M_{l,m}^{(n)}$ have the same asymptotic behavior as  $\tilde{K}_{m}^{(n)}$ and $\tilde{M}_{l,m}^{(n)}$.
 
 The fluctuation theorem of these four quantities $K_{n},M_{l,n},K_{m}^{(n)},M_{l,m}^{(n)}$ have already been obtained in \cite{MR2245368} and \cite{MR3167885}, the fluctuation scale is $n^\alpha$ for $K_{n},M_{l,n}$, and $m^\alpha$ for $K_{m}^{(n)},M_{l,m}^{(n)}$. When we consider scales $n$ and $m$, then
 \begin{equation}
 \lim_{n\to\infty}\frac{K_{n}}{n}= \lim_{n\to\infty}\frac{M_{l,n}}{n}=0,\mbox{ and }  \lim_{m\to\infty}\frac{K_{m}^{(n)}}{m}= \lim_{n\to\infty}\frac{M_{l,m}^{(m)}}{m}=0 \label{LLN}
 \end{equation}
 The corresponding large deviation principle (LDP for short) of (\ref{LLN}) have been discussed thoroughly in \cite{MR1661315},\cite{MR3167885} and \cite{MR3335831}. Moreover, the moderate deviation principle (MDP for short) are also discussed in \cite{MR3850070}. The MDP of $K_{n},M_{l,n}$ has scale $n^{\alpha}\beta_{n}$
  where, however, $\lim_{n\to\infty}\frac{\beta_{n}}{\log^{1-\alpha} n}=\infty$, and there are also similar results for MDP of $K_{m}^{(n)},M_{l,m}^{(n)}$. Usually the MDP scale should be between fluctuation theorem and LDP, i.e.
 $
 n^\alpha \ll n^{\alpha}\beta_{n}\ll n.
 $
 In this sense, the restriction $\lim_{n\to\infty}\frac{\beta_{n}}{\log^{1-\alpha} n}=\infty$ seems unnatural. In this article, we will show that this restriction is actually unnecessary. To establish MDP, we will adopt the scheme in \cite{MR3850070}.
 First, we find the asymptotic log-Laplace transform; then we apply G\"artner-Ellis theorem to complete the proof. Though the moment generating function of $K_{n},M_{l,n},K_{m}^{(n)},M_{l,m}^{(n)}$ have
  explicit series expansion, we can not apply truncation in asymptotic analysis. Therefore, we manage to find a closed integral representation for their moment generating functions.Then we can apply complex asymptotic analysis to these integral representations to 
  get limiting log-Laplace transform. In \cite{MR3850070}, the upper bound and lower bound estimation scheme seems hard to establish limiting log-Laplace transform when $\beta_{n}$ does not satisfy $\lim_{n\to\infty}\frac{\beta_{n}}{\log^{1-\alpha} n}=\infty$.

  This article will focus on the establishment of limiting log-Laplace transform. In section two, we will present the main results and some discussion. In section three, we will provide details of derivations.

\section{Main Results}
By the arguments in \cite{MR3850070}, it only suffices to establish the MPDs of $K_{n},M_{l,n},K_{m}^{(n)},M_{l,m}^{(n)}$ when $\theta=0$. Then, we need to evaluate the following limiting log-Laplace transforms 
\begin{align*}
\psi(\lambda)=&\lim_{n\to\infty}\frac{1}{\beta_{n}^{1/(1-\alpha)}}\log\e\exp\left\{\lambda \frac{\beta_{n}^{1/(1-\alpha)}K_{n}}{n^\alpha\beta_{n}}\right\}\\
\psi_{l}(\lambda)=&\lim_{n\to\infty}\frac{1}{\beta_{n}^{1/(1-\alpha)}}\log\e\exp\left\{\lambda \frac{\beta_{n}^{1/(1-\alpha)} M_{l,n}}{n^\alpha\beta_{n}}\right\}\\
\psi^{(n)}(\lambda)=&\lim_{m\to\infty}\frac{1}{\beta_{m}^{1/(1-\alpha)}}\log\e\left[\exp\left\{\lambda \beta_{m}^{1/(1-\alpha)} \frac{K_{m}^{(n)}}{m^\alpha\beta_{m}}\right\}\mid \{K_{n}=j,N_{n}=n\}\right]\\
\psi_{l}^{(n)}(\lambda)=&\lim_{m\to\infty}\frac{1}{\beta_{m}^{1/(1-\alpha)}}\log\e\left[\exp\left\{\lambda \beta_{m}^{1/(1-\alpha)}\frac{M_{l,m}^{(n)}}{m^\alpha\beta_{m}}\right\}\mid \{K_{n}=j,N_{n}=n\}\right].
\end{align*}
In \cite{MR1661315},\cite{MR3167885},\cite{MR3335831}, the Laplace transform of $K_{n},M_{l,n},K_{m}^{(n)},M_{l,m}^{(n)}$ has explicit power series expansion.
\begin{lemma}
When $\theta=0$, we have
\begin{align*}
\phi_{n}=&\e\exp\left\{\lambda \frac{\beta_{n}^{1/(1-\alpha)}K_{n}}{n^\alpha\beta_{n}}\right\}=\sum_{i=0}^{\infty}y_{n}^{i}\binom{\alpha i+n-1}{n-1}\\
\phi_{l,n}=&\e\exp\left\{\lambda \frac{\beta_{n}^{1/(1-\alpha)} M_{l,n}}{n^\alpha\beta_{n}}\right\}=\sum_{i=0}^{[n/l]}y_{l,n}^i\frac{n}{n-il+i\alpha}\binom{n-il+i\alpha}{n-il}\\
\phi_{m}^{(n)}=&\e\left[\exp\left\{\lambda \beta_{m}^{1/(1-\alpha)} \frac{K_{m}^{(n)}}{m^\alpha\beta_{m}}\right\}\mid \{K_{n}=j,N_{n}=n\}\right]\\
=&(1-y_{m})^{j}\sum_{i=0}^{\infty}(y_{m})^i\frac{j_{(i)}}{i!}\frac{\binom{n+i\alpha+m-1}{n+m-1}}{\binom{n+i\alpha-1}{n-1}}\\
\phi_{l,m}^{(n)}=&\e\left[\exp\left\{\lambda \beta_{m}^{1/(1-\alpha)}\frac{M_{l,m}^{(n)}}{m^\alpha\beta_{m}}\right\}\mid \{K_{n}=j,N_{n}=n\}\right]\\
=&\frac{m!}{n_{(m)}}\sum_{i=0}^{[m/l]}(y_{l,m})^i\binom{j+i-1}{i}\binom{n+m-il+i\alpha-1}{n+i\alpha-1}
\end{align*}
where $y_{n}=1-\exp\{-\lambda\beta_{n}^{\alpha/(1-\alpha)}/n^{\alpha}\}, y_{l,n}=\frac{\alpha(1-\alpha)_{(l-1)}}{l!}\frac{y_{n}}{1-y_{n}}$, and $a_{(k)}=\frac{\Gamma(a+k)}{\Gamma(a)}$.
\end{lemma}
\begin{rmk}
\item[1.] Note that 
\begin{align*}
\psi(\lambda)=&\lim_{n\to\infty}\frac{1}{\beta_{n}^{1/(1-\alpha)}}\log \phi_{n}\\\psi_{l}(\lambda)=&\lim_{n\to\infty}\frac{1}{\beta_{n}^{1/(1-\alpha)}}\log \phi_{l,n}\\
\psi^{(n)}(\lambda)=&\lim_{m\to\infty}\frac{1}{\beta_{m}^{1/(1-\alpha)}}\log \phi_{m}^{(n)}\\
\psi_{l}^{(n)}(\lambda)=&\lim_{m\to\infty}\frac{1}{\beta_{m}^{1/(1-\alpha)}}\log \phi_{l,m}^{(n)}
\end{align*}
\item[2.] Because $1\leq\frac{n}{n-il+i\alpha}\leq\frac{n}{n-\frac{n}{l}l+\frac{n}{l}\alpha}=\frac{l}{\alpha}$, so
$
\psi_{l}(\lambda)=\lim_{n\to\infty}\frac{1}{\beta_{n}^{1/(1-\alpha)}}\log \tilde{\phi}_{l,n},
$
where 
$$\tilde{\phi}_{l,n}=\sum_{i=0}^{[n/l]}(y_{l,n})^i\binom{n-il+i\alpha}{n-il}.
$$
\end{rmk}

Since we can not apply truncation techniques to the above series expansion, we manage to find close integral representations. One can observe that $\binom{k+n-1}{n-1}=\frac{k(k+1)\cdots(k+n-1)}{(n-1)!}=\frac{1}{2\pi \bi}\int_{C}\frac{\xi^{k+n-1}d\xi}{(\xi-1)^n}$, where $C$ is a small circle center at $\xi=1$. Then one can replace all combinatorial coefficients by complex integrals. By appropriately choosing integral contours, we can switch summation and integration. Thus, we have the following integral representations.

\begin{thm}\label{integral}
\begin{align*}
\phi_{n}=&\frac{1}{2\pi \bi}\int_{C}\frac{\xi^{n-1}d\xi}{(\xi-1)^n(1-y_{n}\xi^\alpha)}\\
\tilde{\phi}_{l,n}=&\frac{1}{2\pi \bi}\int_{C}\frac{\xi^{n}d\xi}{(\xi-1)^{n+1}[1-y_{l,n}\xi^{\alpha-l}(\xi-1)^l]}\\
\phi_{m}^{(n)}=&\frac{(1-y_{m})^j}{\binom{n+m-1}{m}}\frac{1}{2\pi \bi}\int_{C}\frac{\xi^{n+m-1}}{(\xi-1)^{m+1}}
\frac{d\xi}{(1-y_{m}\xi^\alpha)^j}\\
\phi_{l,m}^{(n)}=&\frac{m!}{(n)_{(m)}}\frac{1}{2\pi \bi}\int_{C}\frac{\xi^{n+m-1}d\xi}{
(\xi-1)^{m+1}[1-y_{l,m}\xi^{\alpha-l}(\xi-1)^l]^j}
\end{align*}
where the contour $C$ should be a small circle with center $\xi=1$ enclosing the only singularity $\xi=1$.
\end{thm}

To evaluate the above complex integral, we should conduct contour deformation. For $\phi_{n},\phi_{l,n}$, the absolute value of the integrand approach $0$ as $|\xi|\to\infty$; for $\phi_{m}^{(n)},\phi_{l,m}^{(n)}$, the absolute value of the integrand approach $+\infty$ as $|\xi|\to\infty$. Therefore, the integrand is analytic at $\infty$ for $\phi_{n},\phi_{l,n}$, while the integrand is not analytic at $\infty$ for $\phi_{m}^{(n)},\phi_{l,m}^{(n)}$.Therefore, we are going to deform the contour $C$ to Hankel's contour(Figure \ref{dc}) for $\phi_{n},\phi_{l,n}$ and to steepest descent contour(Figure \ref{dc}) for $\phi_{m}^{(n)},\phi_{l,m}^{(n)}$. The precise definition of steepest descent contour is $\arg(\frac{\xi^{n+m-1}}{(\xi-1)^{m+1}})=0$. To choose such a contour is because $$
    \frac{\xi^{n+m-1}}{(\xi-1)^{m+1}}=\left|\frac{\xi^n+m-1}{(\xi-1)^{m+1}}\right|\exp\left\{\bi\arg{\frac{\xi^{n+m-1}}{(\xi-1)^{m+1}}}\right\}.
     $$
If $\arg(\frac{\xi^{n+m-1}}{(\xi-1)^{m+1}})\neq0$, then as $m\to\infty$,
$$
\exp\left\{\bi\arg{\frac{\xi^{n+m-1}}{(\xi-1)^{m+1}}}\right\}
$$
will produce many cancellation, very much the alternating series. The cancellation will prevent us from obtaining the precise asymptotic behavior of the alternating series unless we can somehow transform it into a positive series.
\begin{figure}
  \centering
  \includegraphics[width=0.4\textwidth,height=3in]{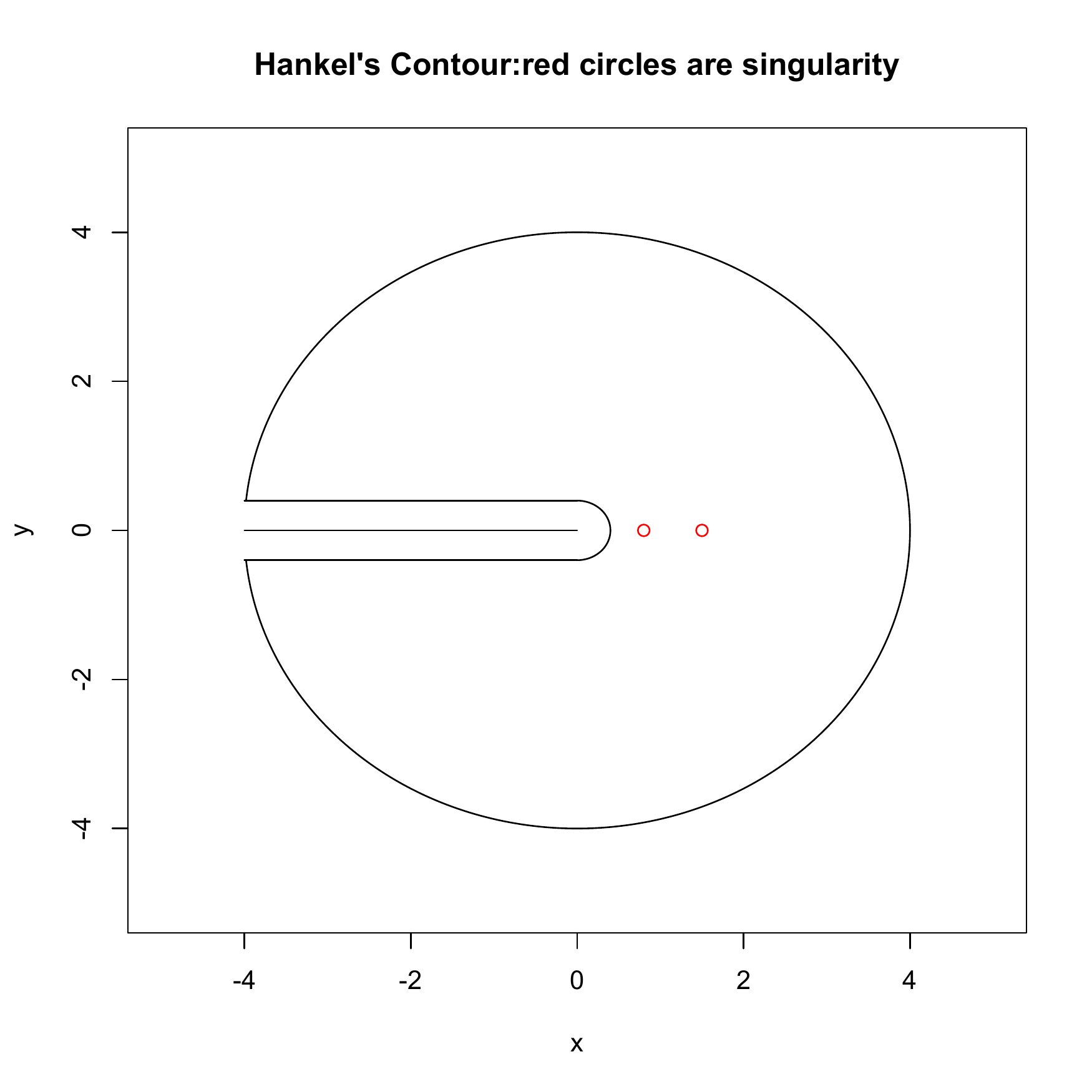}
  \includegraphics[width=0.4\textwidth,height=3in]{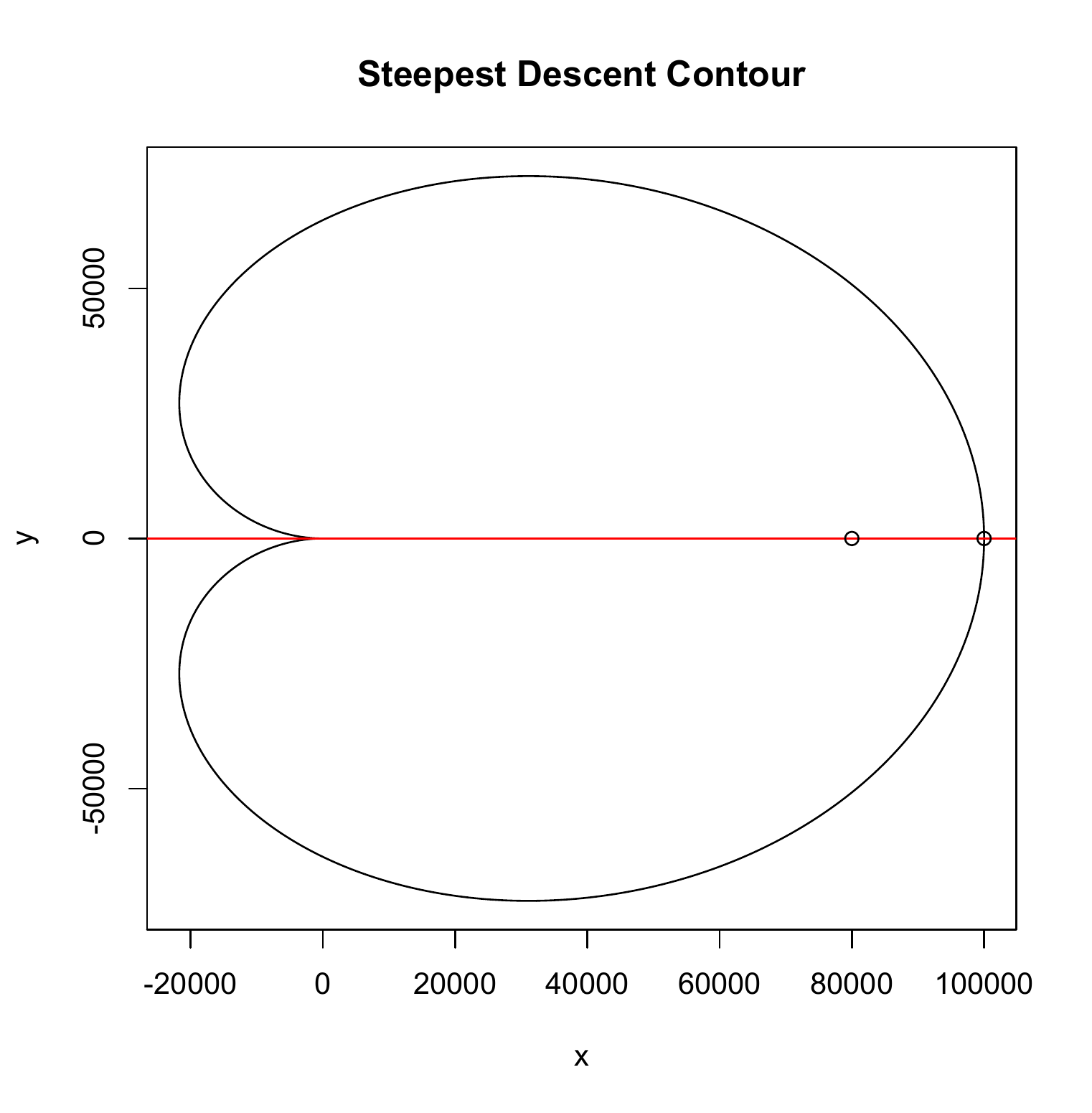} 
  \caption{Contour Deformation}\label{dc}
\end{figure}

In the process of deformation, we are going to cross some singularities, whose residues should be taken into account. For $\phi_{n},\phi_{m}^{(n)}$, the crossed singularity is $\xi=\frac{1}{y_{n}^{1/\alpha}}$. The crossed singularities in $\phi_{l,n},\phi_{l,m}^{(n)}$ are not easy to pinpoint. But one can conclude that the crossed singularities are the zeros of the following equation
$$
1-y_{l,n}\xi^{\alpha-l}(\xi-1)^l=0.
$$
After contour deformation, the complex integrals for $\phi_{n},\tilde{\phi}_{l,n},\phi_{m}^{(n)},\phi_{l,m}^{n}$ can be rewritten as summation of two parts: residues and real integral. The dominant part is the residue. Thus, we obtain the asymptotic beahvior of $\phi_{n},\tilde{\phi}_{l,n}$ as $n\to\infty$, and asymptotic beahvior of $\phi_{m}^{(n)},\phi_{l,m}^{n}$ as $m\to\infty$.
\begin{thm}\label{asymptotic}
The four quantities $\phi_{n},\phi_{l,n},\phi_{m}^{(n)},\phi_{l,m}^{(n)}$ have the following asymptotic behaviour
\begin{itemize}
\item[1] As $n\to\infty$,$\phi_{n}\sim \frac{1}{\alpha(1-y_{n}^{1/\alpha})^n}$, $\phi_{l,n}\sim \frac{1}{\alpha(1-\xi_{n})^{n+l}}$.
\item[2] As $m\to\infty$, 
\begin{align*}
\phi_{m}^{(n)}&\sim \frac{1}{(j-1)!\alpha^j(1-(y_{m}^{(n)})^{1/\alpha})^{m+j}}\frac{1}{[m(y_{m}^{(n)})^{1/\alpha}]^{n-1}}\\
\phi_{l,m}^{(n)}&\sim \frac{1}{(j-1)!\alpha^j(m/\xi_{m})^{n-j}} \frac{1}{(1-1/\xi_{m})^{m+j}}
\end{align*}
\end{itemize}
where $\xi_{n}>1$ is a solution of $1-y_{l,n}\xi^{\alpha-l}(\xi-1)^l=0.$
\end{thm}
Then we can have our limiting log-Laplace transformations of $K_{n},M_{l,n},K_{m}^{(n)},M_{l,m}^{(n)}$, thereby their MDPs.
\begin{thm}\label{mdp}
\item[1]
The limiting log-Laplace transformation of $K_{n},M_{l,n},K_{m}^{(n)},M_{l,m}^{(n)}$ is
$$
\psi_{\lambda}=\psi^{(n)}(\lambda)=\begin{cases}
\lambda^{1/\alpha} & \lambda>0\\
0&\lambda\leq0
\end{cases}
$$
and
$$
\psi_{l}(\lambda)=\psi_{l}^{(n)}(\lambda)=\begin{cases}
(\frac{\alpha(1-\alpha)_{(l-1)}}{l!}\lambda)^{1/\alpha} & \lambda>0\\
0&\lambda\leq0
\end{cases}
$$
\item[2] $K_{n},M_{l,n},K_{m}^{(n)},M_{l,m}^{(n)}$ have $\mathrm{MDP}$s with speed $\beta_{m}^{\frac{1}{1-\alpha}}$ and rate functions $I(x),I_{l}(x)$,\\
$I_{l}(x), I^{(n)}_{l}(x)$, where
    $$
    I(x)=I^{(n)}(x)=\begin{cases}
    (1-\alpha)\alpha^{\alpha/(1-\alpha)}x^{1/(1-\alpha)} & x\leq 0\\
    +\infty& x>0.
    \end{cases}
    $$
    and
    $$
    I_{l}(x)=I^{(n)}_{l}(x)=\begin{cases}
    (1-\alpha)\left(\frac{l!}{(1-\alpha)_{(l-1)}}\right)^{\alpha/(1-\alpha)}x^{1/(1-\alpha)} & x\leq 0\\
    +\infty & x>0.
    \end{cases}
    $$
\end{thm}

\section{Proof of Theorem \ref{integral}}
\subsection{Integral Representation of $\phi_{n}$}
\begin{proof} 
Since
$$
\binom{i\alpha+n-1}{n-1}=\frac{(i\alpha)(i\alpha+1)\cdots(i\alpha+n-1)}{(n-1)!}=\frac{1}{2\pi \bi}\int_{C}\frac{\xi^{i\alpha+n-1}}{(\xi-1)^n}d\xi.
$$
then 
$$
\phi_{n}=\sum_{i=0}^{\infty}\frac{1}{2\pi \bi}\int_{C}(y_{n}\xi^\alpha)^{i}\frac{\xi^{n-1}d\xi}{(\xi-1)^n},
$$
where the contour $C$ is chosen to be a circle with radius $r=|\xi-1|<\frac{1}{2}$. So when $n$ is large enough, $|y_{n}\xi^\alpha|<1$ for $|\xi-1|<\frac{1}{2}$. Then the power series is uniformly convergent in domain $D=\{\xi\in\mathbb{C}\mid|\xi-1|<\frac{1}{2}\}$. After switching of summation and integration, one can have 
$$
\phi_{n}=\frac{1}{2\pi \bi}\int_{C}\frac{\xi^{n-1}d\xi}{(\xi-1)^n}\sum_{i=0}^{\infty}(y_{n}\xi^\alpha)^{i}=\frac{1}{2\pi \bi}\int_{C}\frac{\xi^{n-1}d\xi}{(\xi-1)^n(1-y_{n}\xi^\alpha)}.
$$
\end{proof}
\subsection{Integral Representation of $\phi_{l,n}$}
\begin{proof}
We know
$$
\frac{(n-il+i\alpha)!}{(i\alpha)!(n-il)!}=\frac{(i\alpha+1)\cdots(n-il+i\alpha)}{(n-il)!}
=\frac{1}{2\pi \bi}\int_{C}\frac{\xi^{n-il+i\alpha}}{(\xi-1)^{n-il+1}}d\xi.
$$
where the contour $C$ is a circle contained in domain $D$,
then
\begin{align*}
\tilde{\phi}_{l,n}(\lambda)=&\sum_{i=0}^{[n/l]}y_{l,n}^i\frac{1}{2\pi \bi}\int_{C}\frac{\xi^{n-il+i\alpha}}{(\xi-1)^{n-il+1}}d\xi\\
=&\frac{1}{2\pi \bi}\int_{C}\frac{\xi^{n}}{(\xi-1)^{n+1}}\sum_{i=0}^{[n/l]}[y_{l,n}\xi^{\alpha-l}(\xi-1)^{l}]^i d\xi\\
=&\frac{1}{2\pi \bi}\int_{C}\frac{\xi^{n}}{(\xi-1)^{n+1}}\frac{1-[y_{l,n}\xi^{\alpha-l}(\xi-1)^{l}]^{[n/l]+1}}
{1-y_{l,n}\xi^{\alpha-l}(\xi-1)^l}d\xi\\
=&\frac{1}{2\pi \bi}\int_{C}\frac{\xi^{n}}{(\xi-1)^{n+1}}\frac{1}
{1-y_{l,n}\xi^{\alpha-l}(\xi-1)^l}d\xi\\
&-\frac{1}{2\pi \bi}\int_{C}\frac{\xi^{n}}{(\xi-1)^{n+1}}\frac{[y_{l,n}\xi^{\alpha-l}(\xi-1)^{l}]^{[n/l]+1}}
{1-y_{l,n}\xi^{\alpha-l}(\xi-1)^l}d\xi
\end{align*}
In fact, 
\begin{equation}
\frac{1}{2\pi \bi}\int_{C}\frac{\xi^{n}}{(\xi-1)^{n+1}}\frac{[x\xi^{\alpha-l}(\xi-1)^{l}]^{[n/l]+1}}
{1-x\xi^{\alpha-l}(\xi-1)^l}d\xi=0.\label{claim1}
\end{equation}
Because $n=kl+r$ where $r<l$ and
$$
\frac{\xi^{n}}{(\xi-1)^{n+1}}\frac{[x\xi^{\alpha-l}(\xi-1)^{l}]^{[n/l]+1}}
{1-x\xi^{\alpha-l}(\xi-1)^l}=\frac{\xi^{kl+r+(\alpha-l)(k+1)}}{(\xi-1)^{r+1-l}}\frac{1}{1-x\xi^{\alpha-1}(\xi-1)^l}
$$
We know $\frac{1}{(\xi-1)^{r+1-l}}$ is analytic in $D$ for $r+1-l\leq0$. Therefore, the above function is analytic in $D$.  The claim \ref{claim1} is true. Thus, we have 
$$
\tilde{\phi}_{l,n}=\frac{1}{2\pi \bi}\int_{C}\frac{\xi^{n}d\xi}{(\xi-1)^{n+1}[1-y_{l,n}\xi^{\alpha-l}(\xi-1)^l]}.
$$
\end{proof}
\subsection{Integral Representation of $\phi_{m}^{(n)}$}
\begin{proof}
Since
\begin{align*}
\frac{j_{(i)}}{i!}\frac{\binom{n+i\alpha+m-1}{n+m-1}}{\binom{n+i\alpha-1}{n-1}}
=\frac{1}{\binom{n+m-1}{m}}\frac{(n+i\alpha)\cdots(n+i\alpha+m-1)}{m!}\frac{j(j+1)\cdots(j+i-1)}{i!}
\end{align*}
and
\begin{align*}
\frac{(n+i\alpha)\cdots(n+i\alpha+m-1)}{m!}=&\frac{1}{2\pi \bi}\int_{C_{1}}\frac{\xi_{1}^{n+i\alpha+m-1}}{(\xi_{1}-1)^{m+1}}d\xi_{1}\\
\frac{j(j+1)\cdots(j+i-1)}{i!}=&\frac{1}{2\pi \bi}\int_{C_{2}}\frac{\xi_{2}^{j+i-1}}{(\xi_{2}-1)^{i+1}}d\xi_{2}
\end{align*}
then
\begin{align*}
\frac{j_{(i)}}{i!}\frac{\binom{n+i\alpha+m-1}{n+m-1}}{\binom{n+i\alpha-1}{n-1}}
=\frac{1}{\binom{n+m-1}{m}}\frac{1}{(2\pi \bi)^2}\int_{C_{1}}\int_{C_{2}}\frac{\xi_{1}^{n+i\alpha+m-1}\xi_{2}^{j+i-1}}{(\xi_{1}-1)^{m+1}(\xi_{2}-1)^{i+1}}d\xi_{1}d\xi_{2},
\end{align*}
where $C_{1}$ is a circle centered at $\xi_{1}=1$ with radius $r_{1}\in(\frac{1}{4},\frac{1}{2})$, and $C_{2}$ is a circle centered at $\xi_{2}=1$ with radius $r_{2}\in(\frac{3y_{m}}{2-3y_{m}},\frac{1}{2})$. Then 
\begin{align*}
\phi_{m}^{(n)}=&(1-y_{m}^{(n)})^{j}\sum_{i=0}^{\infty}[y_{m}^{(n)}]^i\frac{1}{\binom{n+m-1}{m}}\frac{1}{(2\pi \bi)^2}\int_{C_{1}}\int_{C_{2}}\frac{\xi_{1}^{n+i\alpha+m-1}\xi_{2}^{j+i-1}}{(\xi_{1}-1)^{m+1}(\xi_{2}-1)^{i+1}}d\xi_{1}d\xi_{2}\\
=&\frac{(1-y_{m}^{(n)})^{j}}{\binom{n+m-1}{m}}\frac{1}{(2\pi \bi)^2}\sum_{i=0}^{\infty}\int_{C_{1}}\int_{C_{2}}\frac{\xi_{1}^{n+m-1}\xi_{2}^{j-1}}{(\xi_{1}-1)^{m+1}(\xi_{2}-1)}(\frac{y_{m}^{(n)}\xi_{1}^{\alpha}\xi_{2}}{\xi_{2}-1})^i d\xi_{1}d\xi_{2}
\end{align*}
One can easily show that 
\begin{equation}
|\frac{y_{m}\xi_{1}^\alpha\xi_{2}}{\xi_{2}-1}|<1.\label{cc}
\end{equation}

Therefore, the power series 
$
\sum_{i=0}^{\infty}(\frac{y_{m}\xi_{1}^{\alpha}\xi_{2}}{\xi_{2}-1})^i
$
is uniformly convergent. We can switch the summation and integration, then
\begin{align*}
\phi_{m}^{(n)}
=&\frac{(1-y_{m})^{j}}{\binom{n+m-1}{m}}\frac{1}{2\pi \bi}\int_{C_{1}}\frac{\xi_{1}^{n+m-1}d\xi_{1}}{(1-y_{m}\xi_{1}^\alpha)(\xi_{1}-1)^{m+1}}\frac{1}{2\pi \bi}\int_{C_{2}}\frac{\xi_{2}^{j-1}d\xi_{2}}{\xi_{2}-\frac{1}{1-y_{m}\xi_{1}^\alpha}}
\end{align*}
By (\ref{cc}) one can check that $\frac{1}{1-y_{m}\xi_{1}^\alpha}$ is the only singularity of the integrand $\frac{\xi_{2}^{j-1}}{\xi_{2}-\frac{1}{1-y_{m}\xi_{1}^\alpha}}$ inside $C_{2}$. Therefore,
$$
\phi_{m}^{(n)}=\frac{(1-y_{m})^{j}}{\binom{n+m-1}{m}}\frac{1}{2\pi \bi}\int_{C_{1}}\frac{\xi_{1}^{n+m-1}d\xi_{1}}{(\xi_{1}-1)^{m+1}(1-y_{m}\xi_{1}^\alpha)^j}.
$$
\end{proof}

\subsection{Integral Representation of $\phi_{l,m}^{(n)}$}
\begin{proof}
\begin{align*}
&\binom{j+i-1}{i}\binom{n+m-il+i\alpha-1}{n+i\alpha-1}\\
=&\frac{(n+i\alpha)\cdots(n+i\alpha+m-il-1)}{(m-il)!}\frac{j(j+1)\cdots(j+i-1)}{i!}\\
=&\frac{1}{2\pi \bi }\int_{C_{1}}\frac{\xi_{1}^{n+i\alpha+m-il-1}}{(\xi_{1}-1)^{m-il+1}}d\xi_{1}\frac{1}{2\pi \bi}\int_{C_{2}}\frac{\xi_{2}^{j+i-1}}{(\xi_{2}-1)^{i+1}}d\xi_{2}\\
=&\frac{1}{(2\pi \bi)^2}\int_{C_{1}}\int_{C_{2}}\frac{\xi_{2}^{j-1}\xi_{1}^{n+m-1}}{(\xi_{2}-1)(\xi_{1}-1)^{m+1}}\left(\frac{\xi_{2}\xi_{1}^{\alpha-l}(\xi_{1}-1)^{l}}{(\xi_{2}-1)}\right)^i d\xi_{1}d\xi_{2},
\end{align*}
where $C_{1},C_{2}$ has chosen to be circles centered at $1$ with radius satisfying
$$
r_{1}=|\xi_{1}-1|\in(\frac{1}{4},\frac{1}{2}), \quad r_{2}=|\xi_{2}-1|\in(\frac{y_{l,m}}{2^\alpha-y_{l,m}},\frac{1}{2})
$$
Then
\begin{align*}
\phi_{l,m}^{(n)}(\lambda)
=&\frac{m!}{(n)_{(m)}}\frac{1}{(2\pi \bi)^2}\int_{C_{1}}\int_{C_{2}}\frac{\xi_{2}^{j-1}\xi_{1}^{n+m-1}}{(\xi_{2}-1)(\xi_{1}-1)^{m+1}}\sum_{i=0}^{[m/l]}\left(y_{l,m}\frac{\xi_{2}\xi_{1}^{\alpha-l}(\xi_{1}-1)^{l}}{(\xi_{2}-1)}\right)^id\xi_{1}d\xi_{2}\\
=&\frac{m!}{(n)_{(m)}}\frac{1}{(2\pi \bi)^2}\int_{C_{1}}\int_{C_{2}}\frac{\xi_{2}^{j-1}\xi_{1}^{n+m-1}}{(\xi_{2}-1)(\xi_{1}-1)^{m+1}}\frac{1}{1-y_{l,m}\frac{\xi_{2}\xi_{1}^{\alpha-l}(\xi_{1}-1)^{l}}{(\xi_{2}-1)}}
d\xi_{1}d\xi_{2}\\
-&\frac{m!}{(n)_{(m)}}\frac{1}{(2\pi \bi)^2}\int_{C_{1}}\int_{C_{2}}\frac{\xi_{2}^{j-1}\xi_{1}^{n+m-1}}{(\xi_{2}-1)(\xi_{1}-1)^{m+1}}\frac{\left[y_{l,m}\frac{\xi_{2}\xi_{1}^{\alpha-l}(\xi_{1}-1)^{l}}{(\xi_{2}-1)}\right]^{[m/l]+1}}{1-y_{l,m}\frac{\xi_{2}\xi_{1}^{\alpha-l}(\xi_{1}-1)^{l}}{(\xi_{2}-1)}}
d\xi_{1}d\xi_{2}
\end{align*}
Similarly, one can show that the second term is $0$. Therefore,
\begin{align*}
\phi_{l,m}^{(n)}=\frac{m!}{(n)_{(m)}}\frac{1}{2\pi \bi}\int_{C_{1}}\frac{\xi_{1}^{n+m-1}d\xi_{1}}{(\xi_{1}-1)^{m+1}(1-y_{l,m}\xi_{1}^{\alpha-l}(\xi_{1}-1)^l)}\frac{1}{2\pi \bi}\int_{C_{2}}\frac{\xi_{2}^{j-1}d\xi_{2}}{\xi_{2}-\frac{1}{1-y_{l,m}\xi_{1}^{\alpha-l}(\xi_{1}-1)^l}}.
\end{align*}
Similarly, one can also show that
$
\xi_{1}=\frac{1}{1-y_{l,m}\xi_{2}^{\alpha-l}(\xi_{2}-1)^l}
$
is the only singularity in $C_{2}$. Then one has
$$
\phi_{l,m}^{(n)}=\frac{m!}{(n)_{(m)}}\frac{1}{2\pi \bi}\int_{C_{1}}\frac{\xi_{2}^{n+m-1}d\xi_{2}}{(\xi_{2}-1)^{m+1}[1-y_{l,m}\xi_{2}^{\alpha-l}(\xi_{2}-1)^l)]^j}.
$$
\end{proof}

\section{Proof of Theorem \ref{asymptotic}}
In the following proofs, we only need to consider the case where $\lambda>0$. Because when $\lambda\leq0$, the log-Laplace transform is a finite number, the limiting log-Laplace transform will be $0$ when $\lambda\leq0$.
\subsection{Asymptotic behavior of $\phi_{n}$}
\begin{proof}
We want to expand $C$ to Hankel contour in Figure \ref{dc} denoted as $C_{middle}$ such that it encloses singular point $\xi=\frac{1}{y_{n}^{1/\alpha}}$. Define
$
f_{n}(\xi)=\frac{\xi^{n-1}}{(\xi-1)^{n}(1-y_{n}\xi^\alpha
)},
$
then
\begin{align*}
\phi_{n}(\lambda)=&\frac{1}{2\pi \bi}\int_{C_{middle}}f_{n}(\xi)d\xi-res_{\xi=\frac{1}{y_{n}^{1/\alpha}}}f_{n}\\
=&\frac{1}{2\pi \bi}\int_{C_{middle}}f_{n}(\xi)d\xi-\lim_{\xi\to\frac{1}{y_{n}^{1/\alpha}}}\frac{\xi^{n-1}(\xi-\frac{1}{y_{n}^{1/\alpha}})}{(\xi-1)^n(1-y_{n}\xi^\alpha)}\\
=&\frac{1}{2\pi \bi}\int_{C_{middle}}f_{n}(\xi)d\xi+\frac{1}{\alpha(1-y_{n}^{1/\alpha})^n}
\end{align*}
For the integral
$$
\frac{1}{2\pi \bi}\int_{C_{middle}}f_{n}(\xi)d\xi,
$$
the contour $C_{middle}$ has four parts: $C_{1},C_{2},C_{3},C_{4}$, where
\begin{itemize}
\item[1] $C_{1}$ is a big circle with the parameter equation
$$
z=Re^{\bi\theta}, -\theta_{0}<\theta<\theta_{0},
$$
where $\theta_{0}=\pi-\arcsin(\frac{r}{R})$ and $0<r<R$.
\item[2] $C_{2}$ is a small circle with the parameter equation
$$
z=re^{\bi\theta}, -\frac{\pi}{2}<\theta<\frac{\pi}{2}
$$
\item[3] $C_{3}$ is a horizontal line on the upper half plane with parameter equation
$$
z=x+\bi r, 0<x<-R.
$$
\item[4] $C_{4}$ is a horizontal line on the lower half plane, the parameter equation is
$$
z=x-\bi r, 0<x<-R
$$
\end{itemize}
The four parts are connected positively according to counterclockwise direction. Now we are going to deform the $C_{middle}$ to $C_{final}$ by letting $R\to\infty$ and $r\to0$. One can show that
$$
\lim_{R\to\infty}\frac{1}{2\pi \bi}\int_{C_{1}}f_{n}(\xi)d\xi=0, \quad \lim_{r\to0}\frac{1}{2\pi \bi}\int_{C_{2}}f_{n}(\xi)d\xi=0
$$
Indeed,$$
|\frac{1}{2\pi \bi}\int_{C_{1}}f_{n}(\xi)d\xi|\leq \frac{R^{n-1}}{(R-1)^n(y_{n}R^\alpha-1)}2\pi R\to0 \mbox{ as } R\to\infty.
$$
and
$$
|\frac{1}{2\pi \bi}\int_{C_{2}}f_{n}(\xi)d\xi|\leq \frac{r^{n-1}}{(1-r)^n(1-y_{n}r^\alpha)}2\pi r\to0 \mbox{ as } r\to0.
$$
In this process, $C_{3}$ becomes a straight line on the upper bank of negative x-axis from $-\infty$ to $0$, and $C_{4}$ becomes a straight line on the lower bank of the negative x-axis from $0$ to $-\infty$. So
$$
\frac{1}{2\pi \bi}\int_{C_{middle}}f_{n}(\xi)d\xi=\frac{1}{2\pi \bi}\int_{-\infty}^{0}f_{n}^{+}(x)dx+\frac{1}{2\pi \bi}\int^{-\infty}_{0}f_{n}^{-}(x)dx
$$
where
$$
f_{n}^{+}(x)=\lim_{r\to0^{+}}f_{n}(x+\bi r)=\lim_{r\to0^{+}}\frac{(x+\bi r)^{n-1}}{(x+ir-1)^n}\frac{1}{1-y_{n}\exp\{\alpha \log |x+\bi r|+\bi\alpha\arg(x+\bi r)\}}.
$$
and
$$
f_{n}(x)=\lim_{r\to0^{+}}f_{n}(x-\bi r)=\lim_{r\to0^{+}}\frac{(x-\bi r)^{n-1}}{(x-\bi r-1)^n}\frac{1}{1-y_{n}\exp\{\alpha \log|x-\bi r|+\bi\alpha\arg(x-\bi r)\}}.
$$
Because
$$
\lim_{r\to0^{+}}\arg(x+\bi r)=\pi,\lim_{r\to0^{+}}\arg(x-\bi r)=-\pi
$$
Then
$$
f_{n}^{+}(x)=\lim_{r\to0^{+}}f_{n}(x+\bi r)=\frac{x^{n-1}}{(x-1)^n}\frac{1}{1-y_{n}|x|^\alpha e^{\bi\alpha\pi}}.
$$
and
$$
f_{n}^{-}(x)=\lim_{r\to0^{+}}f_{n}(x+\bi r)=\frac{x^{n-1}}{(x-1)^n}\frac{1}{1-y_{n}|x|^\alpha e^{-\bi\alpha\pi}}.
$$
Therefore,
\begin{align*}
&\frac{1}{2\pi \bi}\int_{C_{middle}}f_{n}(\xi)d\xi\\
=&\frac{1}{2\pi \bi}\int_{-\infty}^{0}\frac{x^{n-1}}{(x-1)^n}\frac{1}{1-y_{n}|x|^\alpha e^{\bi\alpha\pi}}dx+\frac{1}{2\pi \bi}\int^{-\infty}_{0}\frac{x^{n-1}}{(x-1)^n}\frac{1}{1-y_{n}|x|^\alpha e^{-\bi\alpha\pi}}dx\\
=&\frac{1}{2\pi \bi}\int_{-\infty}^{0}\frac{x^{n-1}}{(x-1)^n}\left[\frac{1}{1-y_{n}|x|^\alpha e^{\bi\alpha\pi}}-\frac{1}{1-y_{n}|x|^\alpha e^{-i\alpha\pi}}\right]dx\\
=&\frac{1}{\pi}\int_{-\infty}^{0}\frac{x^{n-1}}{(x-1)^n} \mathrm{Im}(\frac{1}{1-y_{n}|x|^\alpha e^{\bi\alpha\pi}})dx\\
=&\frac{1}{\pi}\int_{-\infty}^{0}\frac{x^{n-1}}{(x-1)^n}\frac{y_{n}|x|^\alpha\sin\alpha\pi}{y_{n}^2|x|^{2\alpha}-2y_{n}|x|\cos\alpha\pi+1}dx
\end{align*}
Using substitution $u=-\frac{1}{y_{n}^{1/\alpha}x}$, we end up with
$$
\frac{1}{2\pi \bi}\int_{C_{middle}}f_{n}(\xi)d\xi=-\frac{\sin\alpha\pi}{\pi}\int_{0}^{\infty}\frac{du}{(1+y_{n}^{1/\alpha}u)^n u(\frac{1}{u^\alpha}+u^\alpha-2\cos\alpha\pi)}
$$
Note that the function $\frac{1}{u(\frac{1}{u^\alpha}+u^\alpha-2\cos\alpha\pi)}$ is integrable in interval $(0,\infty)$. and $\frac{1}{(1+y_{n}^{1/\alpha}u)^n}=e^{-n\log(1+uy_{n}^{1/\alpha})}\sim e^{-uny_{n}^{1/\alpha}}\to0$. Then by Lebesgue dominant convergent theorem, one can show that $\lim_{n\to\infty}\frac{1}{2\pi \bi}\int_{C_{middle}}f_{n}(\xi)d\xi=0$. But 
$$
\frac{1}{(1-y_{n}^{1/\alpha})^n}=e^{-n\log(1-y_{n}^{1/\alpha})}\sim e^{ny_{n}^{1/\alpha}}\to\infty. 
$$
Then $\phi_{n}\sim \frac{1}{(1-y_{n}^{1/\alpha})^n}$.
\end{proof}

\subsection{Asymptotic behavior of $\phi_{l,n}$}

Before we present the proof, we need the following lemma.
\begin{lemma}\label{ss}
There is a positive integer $N_{0}$, such that, for $n\geq N_{0}$, the equation
$$
1-y_{l,n}\xi^{\alpha-l}(\xi-1)^l=0
$$
 has at most two real solutions and no other solutions.
\begin{itemize}
\item[(a)]When $l$ is even, there are two solutions $\xi_{1,n}<1$ and $\xi_{2,n}>1$. As $n\to\infty$
    $$
    \xi_{1,n}\sim y_{l,n}^{1/(l-\alpha)}\to0\mbox{ and } \xi_{2,n}\sim\frac{1}{y_{l,n}^{1/\alpha}}\to\infty.
    $$
\item[(b)] When $l$ is odd, the only solution is some real number $\xi_{n}>1$. As $n\to\infty$
    $$
 \xi_{n}\sim\frac{1}{y_{l,n}^{1/\alpha}}\to\infty.
    $$
\end{itemize}
\end{lemma}

\begin{proof}
The deformation process is rather similar. First, we want to expand the contour to Hankel contour $C_{middle}$ in Figure \ref{dc} such that the new contour encloses singular points which are the solutions of
 $
 1-y_{l,n}\xi^{\alpha-l}(\xi-1)^l=0.
 $
 Define
 $$
 f_{l,n}(\xi)=\frac{\xi^{n-1}}{(\xi-1)^{n}[1-x_{l,n}\xi^{\alpha-l}(\xi-1)^l]}.
 $$
  By Lemma \ref{ss}, we know $C_{middle}$ enclose two singularities $\xi_{1,n}$ and $\xi_{2,n}$ when $l$ is even; while  $C_{middle}$ enclose  only one singularity $\xi_{n}$ for odd $l$. Similarly, 
 
 $$\begin{cases}
 \tilde{\phi}_{l,n}
 =\frac{1}{2\pi \bi}\int_{C_{middle}}f_{l,n}(\xi)d\xi- res_{\xi=\xi_{1,n}}f_{l,n}- res_{\xi=\xi_{2,n}}f_{l,n}&\mbox{ for even } l\\
 \tilde{\phi}_{l,n}
 =\frac{1}{2\pi \bi}\int_{C_{middle}}f_{l,n}(\xi)d\xi- res_{\xi=\xi_{n}}f_{l,n}&\mbox{ for odd } l 
 \end{cases}
 $$
 By Lemma \ref{ss}, one can check that $res_{\xi=\xi_{1,n}}f_{l,n}\to0$ as $n\to\infty$, but $res_{\xi=\xi_{2,n}}f_{l,n}\sim res_{\xi=\xi_{n}}f_{l,n}\to\infty$. One can also similarly check that $\frac{1}{2\pi \bi}\int_{C_{middle}}f_{l,n}(\xi)d\xi\to0$, then 
 $$
 \phi_{l,n}\sim res_{\xi=\xi_{n}}f_{l,n} \sim \frac{\xi_{n}^{n+l-\alpha+1}}{x_{l,n}(\xi_{n}-1)^{n+l}}\frac{1}{\alpha \xi_{n}+l-\alpha}. 
 $$
 \end{proof}
 
\subsection{Evaluation of $\phi_{m}^{(n)}$}
Before we present the proof, we also need a lemma.
\begin{lemma}\label{residue}
Let $F(\xi)=\sum_{\nu=0}^{\infty}\frac{A_{v+j}}{(v+j)!}\xi^v$ and $G(\xi)=\frac{1}{F(\xi)}$. Then
$$
G^{(k)}(0)=\left[\sum_{s=1}^{k}(-1)^s\sum_{\eta\in\mathcal{P}_{s}(k)}\frac{(j!)^sk!}{(j+\eta_{1})!\cdots(j+\eta_{s})!}\mathcal{A}_{s}^{j}(\eta)\right]\frac{j!}{A_{j}}
$$
where $\mathcal{P}_{s}(k)$ is the set of $s$-component compositions of $k$, and
$$
\mathcal{A}_{s}^{j}(\eta)=\frac{A_{j+\eta_{1}}}{A_{j}}\cdots \frac{A_{j+\eta_{s}}}{A_{j}}.
$$
\end{lemma}

\begin{cor} \label{cresidue}
Let $f(\xi)=\sum_{i=1}^{\infty}\frac{a_{i}}{i!}(\xi-\xi_{m})^i$ has a simple zero $\xi_{m}$, $|\xi_{m}|\to\infty$ and $a_{i}=f^{(i)}(\xi_{m})\sim b_{i}\xi_{m}^{i}$ as $m\to\infty$.  Then 
$$
\left[\left(\frac{\xi-\xi_{m}}{f(\xi)}\right)^j\right]^{(k)}\Bigg|_{\xi=\xi_{m}}\sim D\xi_{m}^{k-j},
$$
where $D=k!\sum_{s=1}^k(-1)^s\sum_{\eta\in\mathcal{P}_{s}(k)}\mathcal{D}_{s}^j(\eta)$, and 
$
\mathcal{D}_{s}^j(\eta)=\frac{j!}{b_{j}}\prod_{u=1}^{s}\frac{b_{j+\eta_{u}}j!}{b_{j}(j+\eta_{u})!}.
$

\end{cor}

\begin{proof}
Define
$
f_{m}^{(n)}(\xi)=\frac{\xi^{n+m-1}}{(\xi-1)^{m+1}}\frac{1}{(1-y_{m}^{(n)}\xi^\alpha)^j}.
$
We want to deform the contour $C$ to the steepest descent contour in Figure \ref{dc}. The steepest descent contour $C_{st}$ is define by 
$$
\arg{\frac{\xi^{n+m-1}}{(\xi-1)^{m+1}}}=0.
$$
Let $\xi=Re^{\bi\theta}$, then
$$
(n+m-1)\theta-(m+1)\arctan\left(\frac{R\sin\theta}{R\cos\theta-1}\right)=0
$$
We end up with the following equation
$$
R(\theta)=\frac{\sin(p\theta)}{\sin(p-1)\theta},-\pi/p\leq\theta<\pi/p,\quad p=\frac{n+m-1}{m+1}.
$$
and
$$
\xi(\theta)=R(\theta)e^{\bi\theta}, -\pi/p\leq\theta<\pi/p.
$$
The two intersections with real axis are $\xi(\pi/p)=0$ and $\xi(0)=\frac{p}{p-1}=\frac{n+m-1}{n-2}$. When $m\to\infty$, $\frac{n+m-1}{n-2}\sim\frac{m}{n-2}$, and the singularity of the integrand is
$$
\xi=\frac{1}{(y_{m})^{1/\alpha}}.
$$
We can show that as $m\to\infty$
$$
\frac{\xi(0)}{\frac{1}{(y_{m})^{1/\alpha}}}\sim\frac{1}{n-2}m (y_{m})^{1/\alpha}\to\infty.
$$
Therefore, the singularity $
\xi=\frac{1}{(y_{m})^{1/\alpha}}
$
is inside $C_{st}$.  So
$$
\phi_{m}^{(n)}=\frac{(1-y_{m})^j}{\binom{n+m-1}{m}}\frac{1}{2\pi \bi}\int_{C_{st}}f_{m}^{(n)}(\xi)d\xi-\frac{(1-y_{m})^j}{\binom{n+m-1}{m}}Res_{
\xi=\frac{1}{(y_{m})^{1/\alpha}}}f_{m}^{n}
$$
where
$$
Res_{
\xi=\frac{1}{(y_{m})^{1/\alpha}}}f_{m}^{n}=\frac{1}{(j-1)!}\frac{d^{j-1}}{d\xi^{j-1}}\Bigg|_{\xi=\frac{1}{(y_{m})^{1/\alpha}}}
\left[\frac{\xi^{n+m-1}(\xi-\frac{1}{(y_{m})^{1/\alpha}})^j}{(\xi-1)^{m+1}(1-y_{m}\xi^\alpha)^j}\right]
$$
We claim that 
\begin{align}
&\frac{(1-y_{m})^j}{\binom{n+m-1}{m}}\frac{1}{2\pi \bi}\int_{C_{st}}f_{m}^{(n)}(\xi)d\xi\to0\label{ain}\\
&\frac{(1-y_{m})^j}{\binom{n+m-1}{m}}Res_{
\xi=-\frac{1}{(y_{m})^{1/\alpha}}}f_{m}^{n}
\sim -\frac{1}{(j-1)!\alpha^j(my_{m}^{1/\alpha})^{n-j}}\frac{1}{(1-y_{m}^{1/\alpha})^{m}}\to\infty\label{are}
\end{align}
Therefore, 
$$
\phi_{m}^{(n)}\sim \frac{1}{(j-1)!\alpha^j(1-(y_{m})^{1/\alpha})^{m+j}}\frac{1}{[m(y_{m})^{1/\alpha}]^{n-1}}.
$$
Now we are going to show the claim (\ref{are}) and claim (\ref{ain}).
First, we know that on $C_{st}$
\begin{align*}
\frac{\xi^{n+m-1}}{(\xi-1)^{m+1}}=&\left|\frac{\xi(\theta)^{n+m-1}}{(\xi(\theta)-1)^{m+1}}\right|=\frac{R(\theta)^{n+m-1}}{(\frac{R(\theta)\sin\theta}{R(\theta)\cos\theta-1})^{\frac{m+1}{2}}}\\
=&\exp\left\{(n+m-1)\log R(\theta)-\frac{m+1}{2}\log \frac{R(\theta)\sin\theta}{R(\theta)\cos\theta-1}\right\}
\end{align*}
Denote
$$
h(\theta)=(n+m-1)\log R(\theta)-\frac{m+1}{2}\log (1+R^2(\theta)-2R(\theta)\cos\theta).
$$
because $R(\theta)$ is an even function, then $h(\theta)$ is also an even function and $R^{'}(\theta)$ is an odd function. 

Now the contour integral becomes
\begin{align*}
 & \frac{1}{2\pi \bi}\int_{C_{st}}\frac{\xi^{n+m-1}}{(\xi-1)^{m+1}}\frac{d\xi}{(1-y_{m}\xi_{1}^\alpha)^j}\\
 =&\frac{1}{2\pi\bi}\int_{-\pi/p}^{\pi/p}\exp\{h(\theta)\}\frac{\xi^{'}(\theta)d\theta}{(1-y_{m}\xi(\theta)^\alpha)^j}
\end{align*}
Because $h(\theta)$ is even, and $\xi^{'}(\theta)=(R^{'}(\theta)+\bi R(\theta))e^{\bi\theta}$,then we know
\begin{align*}
&\frac{1}{2\pi\bi}\int_{-\pi/p}^{\pi/p}\exp\{h(\theta)\}\frac{\xi^{'}(\theta)d\theta}{(1-y_{m}\xi(\theta)^\alpha)^j}\\
=&\frac{1}{2\pi\bi}\int_{0}^{\pi/p}\exp\{h(\theta)\}\left[\frac{\xi^{'}(\theta)}{(1-y_{m}\xi(\theta)^\alpha)^j}+\frac{\xi^{'}(-\theta)}{(1-y_{m}\xi(-\theta)^\alpha)^j}\right]d\theta
\end{align*}
Notice that $\xi(-\theta)=R(\theta)\cos\theta-\bi R(\theta)\sin\theta=\overline{\xi(\theta)}$ and because $R^{'}(\theta)$ is odd function, we have
$$
\xi^{'}(-\theta)=(R^{'}(-\theta)+\bi R(-\theta))e^{-\bi\theta}=(-R^{'}(\theta)+\bi R(\theta))e^{-\bi\theta}
$$
Then
\begin{align*}
&\frac{1}{2\pi\bi}\int_{0}^{\pi/p}\exp\{h(\theta)\}\left[\frac{(R^{'}(\theta)+\bi R(\theta))e^{\bi\theta}}{(1-y_{m}\xi(\theta)^\alpha)^j}+\frac{(-R^{'}(\theta)+\bi R(\theta))e^{-\bi\theta}}{(1-y_{m}\xi(-\theta)^\alpha)^j}\right]d\theta\\
=&\frac{1}{\pi}\int_{0}^{\pi/p}\exp\{h(\theta)\}\left\{R^{'}(\theta)\mathrm{Im}\left[\frac{e^{\bi\theta}}{(1-y_{m}\xi(\theta)^\alpha)^j}\right]+R(\theta)\mathrm{Re}\left[\frac{e^{\bi\theta}}{(1-y_{m}\xi(\theta)^\alpha)^j}\right]\right\}d\theta
\end{align*}
So
\begin{align*}
&\Bigg|\frac{1}{2\pi\bi}\int_{0}^{\pi/p}\exp\{h(\theta)\}\left[\frac{(R^{'}(\theta)+\bi R(\theta))e^{\bi\theta}}{(1-y_{m}\xi(\theta)^\alpha)^j}+\frac{(-R^{'}(\theta)+\bi R(\theta))e^{-\bi\theta}}{(1-y_{m}\xi(-\theta)^\alpha)^j}\right]d\theta\Bigg|\\
\leq&\frac{1}{\pi}\int_{0}^{\pi/p}\exp\{h(\theta)\}\frac{|R(\theta)|+|R^{'}(\theta)|}{|1-y_{m}\xi(\theta)^\alpha|^j}d\theta\\
=&\frac{1}{\pi}\int_{0}^{\pi/p}\exp\{h(\theta)\}\frac{|R(\theta)|+|R^{'}(\theta)|}{|\sqrt{y_{m}^2R^{2\alpha}(\theta)-2y_{m}R^{\alpha}(\theta)\cos\theta\alpha+1}|^j}d\theta
\end{align*}
Note that when $0\leq\theta<\pi$ one can have
$$
R(\theta)\sim\frac{\sin\theta}{\theta}\frac{m+1}{n-2},\quad p=\frac{n+m-1}{m+1}\to1 \mbox{ as } m\to\infty
$$
One can also show that
$$
R^{'}(\theta)=\frac{(1-p)\sin\theta+\cos p\theta\sin(p-1)\theta}{\sin^2(p-1)\theta}\sim\frac{\theta\cos\theta-\sin\theta}{\theta^2}<0,\forall\theta\in(0,\pi)
$$
and 
$$
y_{m}\sim \lambda[\beta_{m}/m^{1-\alpha}]^{\alpha/(1-\alpha)},  y_{m}R^\alpha(\theta)\sim \lambda\frac{\sin\theta}{(n-2)\theta}\beta_{m}^{\alpha/(1-\alpha)}\to\infty
$$
Then when $m$ is large enough $R(\theta)$ is strictly decreasing. So there is a unique $0<\theta_{m}<\frac{\pi}{p}$, such that $\frac{\pi}{p}-\theta_{m}\sim\frac{1}{\beta_{m}^{2/(1-\alpha)}}$. Then  
$$
R(\theta_{m})=\frac{\sin(\pi-p\theta_{m})}{[(p-1)\theta_{m}]\sin(p-1)\theta_{m}/[(p-1)\theta_{m}]}\sim\frac{\pi-p\theta_{m}}{(p-1)\theta_{m}}\sim \frac{m(\pi/p-\theta_{m})}{(n-2)\pi}\sim \frac{1}{(n-2)\pi}\frac{m}{\beta_{m}^{2/(1-\alpha)}}
$$
 Then 
$$
\frac{1}{\pi}\int_{0}^{\pi/p}\exp\{h(\theta)\}\frac{|R(\theta)|+|R^{'}(\theta)|}{|\sqrt{y_{m}^2R^{2\alpha}(\theta)-2y_{m}R^{\alpha}(\theta)\cos\theta\alpha+1}|^j}d\theta=\mathrm{I}+\mathrm{II}
$$
where 
\begin{align*}
\mathrm{I}=&\int_{\theta_{m}}^{\pi/p}\exp\{h(\theta)\}\frac{|R(\theta)|+|R^{'}(\theta)|}{|\sqrt{y_{m}^2R^{2\alpha}(\theta)-2y_{m}R^{\alpha}(\theta)\cos\theta\alpha+1}|^j}d\theta\\
\mathrm{II}=&\int_{0}^{\theta_{m}}\exp\{h(\theta)\}\frac{|R(\theta)|+|R^{'}(\theta)|}{|\sqrt{y_{m}^2R^{2\alpha}(\theta)-2y_{m}R^{\alpha}(\theta)\cos\theta\alpha+1}|^j}d\theta
\end{align*}
When $\theta\in[\theta_{m},\pi/p]$, we have $|\sqrt{y_{m}^2R^{2\alpha}(\theta)-2y_{m}R^{\alpha}(\theta)\cos\theta\alpha+1}|\geq\sqrt{\sin^2\theta\alpha}\geq|\sin(\pi\alpha)|$; and when $\theta\in[0,\theta_{m}]$ we have $|\sqrt{y_{m}^2R^{2\alpha}(\theta)-2y_{m}R^{\alpha}(\theta)\cos\theta\alpha+1}|\geq|\sqrt{y_{m}^2R^{2\alpha}(\theta)-2y_{m}R^\alpha(\theta)+1}|\geq (y_{m}R^{\alpha}(\theta_{m})-1)$. Moreover, one can show that there exists constant $M>0$ such that
\begin{equation}
|R^{'}(\theta)|\leq M,\quad |e^{h(\theta)-h(0)}|\leq M \label{bound}
\end{equation}
Indeed, 
\begin{align*}
&h(\theta)-h(0)=(n+m-1)\log\frac{R(\theta)}{R(0)}-\frac{m+1}{2}\log\frac{1+R^2(\theta)-2R(\theta)\cos\theta}{(R(0)-1)^2}\\
=&(n+m-1)\log\frac{R(\theta)}{R(0)}-\frac{m+1}{2}\log(\frac{R(\theta)}{R(0)})^2-\frac{m+1}{2}\log\frac{(1-2\frac{\cos\theta}{R(\theta)}+\frac{1}{R^2(\theta)})}{(1-\frac{1}{R(0)})^2}\\
=&(n-2)\log \frac{R(\theta)}{R(0)}-\frac{m+1}{2}\log(p^2-2\frac{p^2\cos\theta}{R(\theta)}+\frac{p^2}{R^2(\theta)})\\
=&(n-2)\log \frac{R(\theta)}{R(0)}-\frac{m+1}{2}\log(1+2(p-1)+(p-1)^2-2\frac{p^2\cos\theta}{R(\theta)}+\frac{p^2}{R^2(\theta)})\\
\to&(n-2)\log \frac{\sin\theta}{\theta}-\lim_{m\to\infty}\frac{m+1}{2}[2(p-1)+(p-1)^2-2\frac{p^2\cos\theta}{R(\theta)}+\frac{p^2}{R^2(\theta)}]\\
=&(n-2)\log \frac{\sin\theta}{\theta}+(n-2)\frac{\theta\cos\theta}{\sin\theta}.
\end{align*}
so $\exp\{h(\theta)-h(0)\}\leq M$, 
and 
\begin{align*}
\exp\{h(0)\}=&\exp\{(n+m-1)\log R(0)-(m+1)\log(R(0)-1)\}\\
=&(\frac{n+m-1}{n-2})^{n-2}(\frac{n+m-1}{m+1})^{m+1}\\
\sim&  (\frac{e}{n-2})^{n-2}m^{n-2}
\end{align*}
So 
\begin{align*}
\mathrm{I}&\leq \frac{M}{\sin^j\pi\alpha}(R(\theta_{m})+M)(\frac{\pi}{p}-\theta_{m})\sim C\frac{m}{\beta_{m}^{4/(1-\alpha)}}\\
\mathrm{II}&\leq \pi M\frac{R(0)+M}{(y_{m}R^{\alpha}(\theta_{m})-1)^j}\sim C \frac{m}{\beta_{m}^{\frac{2j}{1-\alpha}}}
\end{align*}

Then 
\begin{align*}
&|\frac{(1-y_{m})^j}{\binom{n+m-1}{m}}\frac{1}{2\pi \bi}\int_{C_{st}}f_{m}^{(n)}(\xi)d\xi|\\
\leq&\frac{(1-y_{m})^j}{\binom{n+m-1}{m}}\exp\{h(0)\}\frac{1}{\pi}\int_{0}^{\pi/p}\exp\{h(\theta)-h(0)\}\frac{|R(\theta)|+|R^{'}(\theta)|}{|\sqrt{y_{m}^2R^{2\alpha}(\theta)-2y_{m}R^{\alpha}(\theta)\cos\theta\alpha+1}|^j}d\theta\\
&\sim C\frac{1}{m^{n-1}}(\frac{e}{n-2})^{n-2}m^{n-2}\left[\frac{m}{\beta_{m}^{4\alpha/(1-\alpha)}}+\frac{m}{\beta_{m}^{2\alpha j/(1-\alpha)}}\right]\to0.
\end{align*}
Thus, claim (\ref{ain}) is true. 

To show claim (\ref{are}), we need Lemma \ref{residue}. Using Leibnitz law, we will have
\begin{equation}
\sum_{k=0}^{j-1}\binom{j-1}{k}\left(\frac{\xi^{n+m-1}}{(\xi-1)^{m+1}}\right)^{(j-k-1)}\left[\left(\frac{\xi-1/y_{m}^{1/\alpha}}{1-y_{m}\xi^\alpha}\right)^{j}\right]^{(k)}\label{lk}
\end{equation}
Because $\xi=\frac{1}
{y_{m}^{1/\alpha}}\to\infty$, so by Leibnitz's formula one can also have
\begin{equation}
\left(\frac{\xi^{n+m-1}}{(\xi-1)^{m+1}}\right)^{(j-k-1)}\Bigg|_{\xi=\frac{1}{y_{m}^{1/\alpha}}}\sim (-1)^{j-k-1}(\frac{\xi}{\xi-1})^{m+j-k}\xi^{n-j+k-1}(m/\xi)^{j-k-1}\Bigg|_{\xi=\frac{1}{y_{m}^{1/\alpha}}}.\label{sl}
\end{equation}
 Therefore,
\begin{align*}
\left[\frac{\xi^{m+n-1}}{(\xi-1)^{m+1}}\right]^{(j-k-1)}\Bigg|_{\xi=\frac{1}{y_{m}^{1/\alpha}}}
&\sim(-1)^{j-k-1}\frac{(m (y_{m})^{1/\alpha})^{j-k-1}}{(1-y_{m}^{1/\alpha})^{m+j-k}}\frac{1}{y_{m}^{(n-j+k-1)/\alpha}}\\
&\sim (-1)^{j-k-1}\frac{(m y_{m}^{1/\alpha})^{j-k-1}}{(1-y_{m}^{1/\alpha})^{m}}\frac{1}{(y_{m}^{1/\alpha})^{n-j+k-1}}
\end{align*}
Since
$$
f(\xi)=\frac{1-y_{m}\xi^\alpha}{\xi-1/y_{m}^{1/\alpha}}=\sum_{i=1}^{\infty}\frac{a_{i}}{i!}(\xi-1/y_{m}^{1/\alpha})^{i-1},
$$
where
$$
a_{i}=f^{(i)}(1/y_{m}^{1/\alpha})=-\alpha_{(i)}y_{m}(1/y_{m}^{1/\alpha})^{\alpha-i}\sim(-1)^{i}\alpha (y_{m})^{\frac{i}{\alpha}}(1-\alpha)_{(i-1)}
$$
then we know the coefficients in $F(\xi)=f(\xi)^j=\sum_{v=0}\frac{A_{v+j}}{A_{j}}(\xi-1/y_{m}^{1/\alpha})^{v}$ satisfies the following
\begin{align*}
&A_{v+j}\sim(-1)^{v+j}(y_{m}^{1/\alpha})^{v+j}\alpha^j\sum_{i_{1}+\cdots+i_{j}=v+j}\binom{j+v}{i_{1}\cdots i_{j}}(1-\alpha)_{(i_{1}-1)}\cdots(1-\alpha)_{(i_{j}-1)}
\end{align*}
Then $G(\xi)=\frac{1}{F(\xi)}$, by Lemma \ref{residue} and Corollary \ref{cresidue} one can have
$$
G^{(k)}(1/y_{m}^{1/\alpha})\sim (-1)^{k-j}(y_{m}^{1/\alpha})^{k-j}\frac{1}{\alpha^j}D
$$
where $D=k!\sum_{s=1}^{k}(-1)^s\sum_{\eta\in\mathcal{P}_{s}(k)}\mathcal{D}_{s}^j(\eta)$,
$
\mathcal{D}_{s}^{j}=C_{\eta_{1}}(\alpha)\cdots C_{\eta_{s}}(\alpha)
$
and
$$
C_{p}(\alpha)=\sum_{i_{1}+\cdots+i_{j}=p+j}\frac{(1-\alpha)_{(i_{1}-1)}}{i_{1}!}\cdots\frac{(1-\alpha)_{(i_{j}-1)}}{i_{j}!}.
$$

Therefore, in (\ref{lk}) we know for given $0\leq k\leq j-1$\begin{align*}
&\left(\frac{\xi^{n+m-1}}{(\xi-1)^{m+1}}\right)^{(j-1-k)}\left(\frac{(\xi-1/y_{m}^{1/\alpha})^j}{(1-y_{m}\xi^\alpha)^j}\right)^{(k)}\Bigg|_{\xi=1/y_{m}^{1/\alpha}}
\sim C\frac{(my_{m}^{1/\alpha})^{j-k-1}}{(y_{m}^{1/\alpha})^{n-1}}\frac{1}{(1-y_{m}^{1/\alpha})^{m}}
\end{align*}
Because $my_{m}^{1/\alpha}\to\infty$, then the leading term is the summand where $k=0$, $i.e.$
\begin{align*}
&\left(\frac{\xi^{n+m-1}}{(\xi-1)^{m+1}}\right)^{(j-1)}\left(\frac{\xi-1/y_{m}^{1/\alpha}}{1-y_{m}\xi^\alpha}\right)^{j}\Bigg|_{\xi=1/y_{m}^{1/\alpha}}
\sim-\frac{(my_{m}^{1/\alpha})^{j-1}}{(y_{m}^{1/\alpha})^{n-1}\alpha^j}\frac{1}{(1-y_{m}^{1/\alpha})^{m}}
\end{align*}
Therefore,
\begin{align*}
&\frac{(1-y_{m})^j}{\binom{n+m-1}{m}}Res_{
\xi=\frac{1}{(y_{m})^{1/\alpha}}}f_{m}^{n}
\sim-\frac{1}{(j-1)!\alpha^j(my_{m}^{1/\alpha})^{n-j}}\frac{1}{(1-y_{m}^{1/\alpha})^{m}}.
\end{align*}
\end{proof}

\subsection{Asymptotic of $\phi_{l,m}^{(n)}$}
Define
$
f_{l,m}^{(n)}=\frac{\xi^{n+m-1}}{(\xi-1)^{m+1}}\frac{1}{[1-y_{l,m}\xi^{\alpha-l}(\xi-1)^l]^j}.
$
Similarly, we will deform the contour to steepest descent contour $C_{st}$. Then the integral can be written as summation of two parts: residue part and integral part
$$
\phi_{l,m}^{(n)}=\frac{1}{\binom{m+n-1}{m}}\frac{1}{2\pi\bi}\int_{C_{st}}f_{l,m}^{(n)}(\xi)d\xi-\frac{1}{\binom{m+n-1}{m}}Res_{\xi=\xi_{n}}f_{l,m}^{(n)}.
$$
Similarly, one can show that $\frac{1}{\binom{m+n-1}{m}}\int_{C_{st}}f_{l,m}^{(n)}(\xi)d\xi\to0$. The proof can be carried out quite similarly. Note that 
\begin{align*}
\frac{1}{\binom{m+n-1}{m}}\frac{1}{2\pi\bi}\int_{C_{st}}f_{l,m}^{(n)}(\xi)d\xi=&\frac{1}{\binom{m+n-1}{m}}\frac{1}{2\pi\bi}\int_{-\pi/p}^{\pi/p}\exp\{h(\theta)\}\frac{\xi^{'}(\theta)d\theta}{[1-y_{l,m}\xi^{\alpha-l}(\theta)(\xi(\theta)-1)^l]^j}\\
=&\frac{1}{\binom{m+n-1}{m}}\frac{1}{2\pi\bi}\Bigg[\int_{0}^{\pi/p}e^{h(\theta)}\frac{[R^{'}(\theta)+iR(\theta)]e^{i\theta}d\theta}{[1-y_{l,m}R^{\alpha}(\theta)e^{i\alpha\theta}(1-\frac{1}{R(\theta)}e^{-\bi \theta})^l]^j}\\
&-\int^{0}_{-\pi/p}e^{h(\theta)}\frac{[R^{'}(\theta)+iR(\theta)]e^{i\theta}d\theta}{[1-y_{l,m}R^{\alpha}(\theta)e^{i\alpha\theta}(1-\frac{1}{R(\theta)}e^{-\bi \theta})^l]^j}\Bigg]\\
=&\frac{1}{\binom{m+n-1}{m}}\frac{1}{\pi}\int_{0}^{\pi/p}e^{h(\theta)}\Bigg[R^{'}(\theta)\mathrm{Im}\frac{e^{i\theta}}{[1-y_{l,m}R^{\alpha}(\theta)e^{i\alpha\theta}(1-\frac{1}{R(\theta)}e^{-\bi \theta})^l]^j}\\
&+R(\theta)\mathrm{Re}\frac{e^{i\theta}}{[1-y_{l,m}R^{\alpha}(\theta)e^{i\alpha\theta}(1-\frac{1}{R(\theta)}e^{-\bi \theta})^l]^j}\Bigg]d\theta
\end{align*}
Moreover, 
$$
|\frac{1}{\binom{m+n-1}{m}}\frac{1}{2\pi\bi}\int_{C_{st}}f_{l,m}^{(n)}(\xi)d\xi|\leq \frac{1}{\binom{m+n-1}{m}}\frac{1}{\pi}\int_{0}^{\pi/p}\exp\{h(\theta)\}\frac{(|R(\theta)|+|R^{'}(\theta)|)d\theta}{[|1-y_{l,m}R^{\alpha}(\theta)e^{\bi\alpha\theta}(1-\frac{1}{R(\theta)}e^{-\bi \theta})^l|]^j}
$$
We can rewrite 
$$|1-y_{l,m}R^{\alpha}(\theta)e^{\bi\alpha\theta}(1-\frac{1}{R(\theta)}e^{-\bi \theta})^l|=y_{l,m}^2t^2(\theta)-2y_{l,m}t(\theta)\cos B(\theta)+1$$
where 
$$
t(\theta)=(1-\frac{2\cos\theta}{R(\theta)}+\frac{1}{R^2(\theta)})^{\frac{l}{2}}R^{\alpha}(\theta)
$$
and
$$
B(\theta)=l\arctan\frac{\sin\theta}{R(\theta)-\cos\theta}+\alpha\theta.
$$
Now we can choose $0\leq \theta_{m}\leq \frac{\pi}{p}$ such that $\frac{\pi}{p}-\theta_{m}\sim \frac{1}{\beta_{m}^{2/(1-\alpha)}}$. Then 
\begin{align*}
&\frac{1}{\binom{m+n-1}{m}}\frac{1}{\pi}\int_{0}^{\pi/p}\exp\{h(\theta)\}\frac{(|R(\theta)|+|R^{'}(\theta)|)d\theta}{[\sqrt{y_{l,m}^2t^2(\theta)-2y_{l,m}t(\theta)\cos B(\theta)+1}]^j}=\mathrm{I}+\mathrm{II}
\end{align*}
where
\begin{align*}
\mathrm{I}=&\frac{1}{\binom{m+n-1}{m}}\frac{1}{\pi}\int_{\theta_{m}}^{\pi/p}\exp\{h(\theta)\}\frac{(|R(\theta)|+|R^{'}(\theta)|)d\theta}{[\sqrt{y_{l,m}^2t^2(\theta)-2y_{l,m}t(\theta)\cos B(\theta)+1}]^j}\\
\mathrm{II}=&\frac{1}{\binom{m+n-1}{m}}\frac{1}{\pi}\int_{0}^{\theta_{m}}\exp\{h(\theta)\}\frac{(|R(\theta)|+|R^{'}(\theta)|)d\theta}{[\sqrt{y_{l,m}^2t^2(\theta)-2y_{l,m}t(\theta)\cos B(\theta)+1}]^j}\end{align*}
Note that $t(\theta)\sim R^{\theta}(\theta),\forall\theta\in(0,\theta_{m}), B(\theta)\to\alpha\theta, \forall \theta\in(0,\pi/p)$. Then we can show that $\mathrm{I},\mathrm{II}\to0$, the proof of which are exactly the same as $\phi_{m}^{(n)}$.

Now we are going to consider the asymptotic behaviour of residuals. When $l$ is even, then $f_{l,m}^{(n)}$ will have two residuals at $\eta_{1,m},\eta_{2,m}$. Applying Leibnitz's formula, we have 
\begin{equation}
\sum_{k=0}^{j-1}\binom{j-1}{k}\left(\frac{\xi^{n+m-1}}{(\xi-1)^{m+1}}\right)^{(j-k-1)}\left(\frac{(\xi-\xi_{m})^j}{[1-y_{l,m}\xi^{\alpha-l}(\xi-1)^l]^j}\right)^{(k)}.\label{lm}
\end{equation}

By assumption $\xi_{1,m}\sim\xi_{1,m}^{l-\alpha}\to0$ and $\xi_{2,m}\sim\frac{1}{(y_{l,m})^{1/\alpha}}\to\infty$,we know
$$
(\frac{\xi^{n+m-1}}{(\xi-1)^{m+1}})^{(j-k-1)}\Bigg|_{\xi=\eta_{1,m}}\sim(-1)^{m-1}m^{j-k-1}\eta_{1,m}^{n+m-j+k}
$$
and
\begin{align*}
(\frac{\xi^{n+m-1}}{(\xi-1)^{m+1}})^{(j-k-1)}\Bigg|_{\xi=\xi_{2,m}}\sim&(-1)^{j-k-1}(\frac{\xi_{2,m}}{\xi_{2,m}-1})^{m+j-k}
(m/\xi_{2,m})^{j-k-1}\xi_{2,m}^{n-j+k-1}.
\end{align*}
By mathematical induction, we can get the expression of coefficients in power series
$$
g(\xi)=\frac{1-x_{l,m}\xi^{\alpha-l}(\xi-1)^l}{\xi-\xi_{1,m}}=\sum_{i=1}^{\infty}\frac{b_{i}^1}{i!}(\xi-\xi_{1,m})^{i-1}
$$
and
$$
g(\xi)=\frac{1-x_{l,m}\xi^{\alpha-l}(\xi-1)^l}{\xi-\xi_{2,m}}=\sum_{i=1}^{\infty}\frac{b_{i}^2}{i!}(\xi-\xi_{2,m})^{i-1},
$$
where
$$
b^{1}_{i}=g^{(i)}(\xi_{1,m})=-y_{l,m}\alpha_{[k]}\xi_{1,m}^{\alpha-2k}(\frac{\xi_{1,m}-1}{\xi_{1,m}})^{l-k}Q^{k}_{1}(\xi_{1,m};\alpha,l)
$$
$$
b^{2}_{i}=g^{(i)}(\xi_{2,m})=-y_{l,m}\alpha_{[k]}\xi_{2,m}^{\alpha-2k}(\frac{\xi_{2,m}-1}{\xi_{2,m}})^{l-k}Q^{k}_{1}(\xi_{2,m};\alpha,l)
$$
and $Q_{1}^{1}(\xi;\alpha,l)=\alpha(\xi-1)+l$,
\begin{align*}
Q^{k}_{1}(\xi;\alpha,l)=&\xi(\xi-1)(Q^{k-1}_{1}(\xi;\alpha,l))^{'}
+[(\alpha-2k+2)(\xi-1)+l-k+1]Q^{k-1}_{1}(\xi;\alpha,l).
\end{align*}
Then by mathematical induction again, one can show that
\begin{align*}
b^{1}_{i}\sim&(-1)^{l-i-1}y_{l,m}\alpha_{[i]}(l-\alpha)_{(i)}\xi_{1,m}^{\alpha-l-i}
=(-1)^{l-2}y_{l,m}\alpha(1-\alpha)_{(i-1)}(l-\alpha)_{(i)}\xi_{1,m}^{\alpha-l-i}\\
b^{2}_{i}\sim&-\alpha_{[i]}y_{l,m}\xi^{\alpha-i}\Big|_{\xi=\xi_{2,m}}=(-1)^i\alpha(1-\alpha)_{(i-1)}y_{l,m}\xi_{2,m}^{\alpha-i}
\end{align*}
Then due to $\xi_{1,m}\sim (y_{l,m})^{1/(l-\alpha)}\to0$ and $\xi_{2,m}\sim 1/(y_{l,m})^{1/\alpha}\to\infty$, we have $F_{1}(\xi)=g_{1}^j(\xi)$ and $F_{2}(\xi)=g_{1}^j(\xi)$ can be expanded as
$$
F_{1}(\xi)=\sum_{v=0}^{\infty}\frac{A_{v+j}^{1}}{(v+j)!}(\xi-\xi_{1,m})^{v},\quad F_{2}(\xi)=\sum_{v=0}^{\infty}\frac{A_{v+j}^{2}}{(v+j)!}(\xi-\xi_{2,m})^{v}
$$
where as $m\to\infty$
\begin{align*}
A_{v+j}^{1}=&\sum_{i_{1}+\cdots+i_{j}=v+j}\binom{v+j}{i_{1}\cdots i_{j}}b^{1}_{i_{1}}\cdots b^{1}_{i_{j}}
\sim (-1)^{(l-2)j}(y_{l,m})^j \alpha^j(v+j)!\eta_{1,m}^{(\alpha-l)j-v-j}E_{v}^1(\alpha)\\
\sim&(-1)^{(l-2)j}\alpha^j(v+j)!\eta_{1,m}^{-v-j}E_{v}^1(\alpha),
\end{align*}
\begin{align*}
A_{v+j}^{2}=&\sum_{i_{1}+\cdots+i_{j}=v+j}\binom{v+j}{i_{1}\cdots i_{j}}b^{2}_{i_{1}}\cdots b^{2}_{i_{j}}
\sim(-1)^{v+j}\alpha^j(y_{l,m})^j(v+j)!\xi_{2,m}^{\alpha j-v-j}E_{v}^2(\alpha)\\
\sim&(-1)^{v+j}\alpha^j(v+j)!\xi_{2,m}^{-v-j}E_{v}^2(\alpha)
\end{align*}
and
\begin{align*}
E_{v}^1(\alpha)=&\sum_{i_{1}+\cdots+i_{j}=v+j}\prod_{u=1}^{j}\frac{(1-\alpha)_{(i_{u}-1)}(l-\alpha)_{(i_{u})}}{i_{u}!},\quad 
E_{v}^2(\alpha)=\sum_{i_{1}+\cdots+i_{j}=v+j}\prod_{u=1}^{j}\frac{(1-\alpha)_{(i_{u}-1)}}{i_{u}!}.
\end{align*}
Then $G(\xi)=\frac{1}{F(\xi)}$, by Lemma \ref{residue} we have
\begin{align*}
\mathcal{A}_{s}^{1,j}(\eta)=&\prod_{v=1}^{s}\frac{A^1_{\eta_{v}+j}}{A^1_{j}}\sim \frac{(\eta_{1}+j)!\cdots(\eta_{1}+j)!}{(j!)^s}\eta_{1,m}^{-k}\mathcal{D}_{s}^1(\eta)
\end{align*}
and
\begin{align*}
\mathcal{A}_{s}^{2,j}(\eta)=&\prod_{v=1}^{s}\frac{A^2_{\eta_{v}+j}}{A^2_{j}}\sim(-1)^k\frac{(\eta_{1}+j)!\cdots(\eta_{1}+j)!}{(j!)^s}\eta_{2,m}^{-k}\mathcal{D}^2_{s}(\eta)
\end{align*}
where
\begin{align*}
\mathcal{D}_{s}^1(\eta)=&\sum_{\eta\in\mathcal{P}_{s}(k)}\prod_{v=1}^{s}E_{v}^1(\alpha)\quad
\mathcal{D}_{s}^2(\eta)=\sum_{\eta\in\mathcal{P}_{s}(k)}\prod_{v=1}^{s}E_{v}^2(\alpha)
\end{align*}
By Lemma \ref{residue} and Corollary \ref{cresidue}, we know $G^{(k)}_{1}(\eta_{1,m})$ and $G^{(k)}_{2}(\eta_{2,m})$ satisfies
\begin{align*}
&G^{(k)}(\eta_{1,m})\sim \frac{j!}{A_{j}^1}\eta_{1,m}^{-k}E_{1}\sim (-1)^{(l-2)j}(\frac{1}{\alpha})^j\eta_{1,m}^{j-k}E_{1},\\
&G^{(k)}(\eta_{2,m})\sim (-1)^{k}\frac{j!}{A_{j}^2}\eta_{2,m}^{-k}E_{2}\sim (-1)^{k-j}(\frac{1}{\alpha})^j\eta_{2,m}^{j-k}E_{2}
\end{align*}
where
\begin{align*}
E_{1}=&k!\sum_{s=1}^{k}(-1)^s\mathcal{D}^1_{s}(\eta),\quad
E_{2}=k!\sum_{s=1}^{k}(-1)^s\mathcal{D}^2_{s}(\eta)
\end{align*}

Thus, we know
$$
(\frac{\xi^{n+m-1}}{(\xi-1)^{m+1}})^{(j-k-1)}\left(\frac{(\xi-\xi_{m})^j}{[1-y_{l,m}\xi^{\alpha-l}(\xi-1)^l]^j}\right)^{(k)}\Bigg|_{\xi=\xi_{1,m}}\sim(-1)^{m+(l-2)j-1}(\frac{1}{\alpha})^jE_{1}m^{j-k-1}\xi_{1,m}^{n+m}
$$
and
$$
(\frac{\xi^{n+m-1}}{(\xi-1)^{m+1}})^{(j-k-1)}\left(\frac{(\xi-\xi_{m})^j}{[1-y_{l,m}\xi^{\alpha-l}(\xi-1)^l]^j}\right)^{(k)}\Bigg|_{\xi=\xi_{2,m}}\sim(-1)^{m+k-j-1}(\frac{1}{\alpha})^jE_{2}m^{j-k-1}\xi_{2,m}^{n+m}.
$$
Therefore, the leading terms in (\ref{lk}) are the ones when $k=0$. So
\begin{align*}
&\frac{d^{j-1}}{d\xi^{j-1}}\left[\frac{\xi^{n+m-1}}{(\xi-1)^{m+1}}\frac{(\xi-\xi_{1,m})^j}{[1-y_{l,m}\xi^{\alpha-l}(\xi-1)^l]^j}\right]\Bigg|_{\xi=\xi_{1,m}}
\sim(-1)^{-m+(l-2)j-1}(1/\alpha)^j\xi_{1,m}^{n+m}m^{j-1}
\end{align*}
\begin{align*}
&\frac{d^{j-1}}{d\xi^{j-1}}\left[\frac{\xi^{n+m-1}}{(\xi-1)^{m+1}}\frac{(\xi-\xi_{2,m})^j}{[1-y_{l,m}\xi^{\alpha-l}(\xi-1)^l]^j}\right]\Bigg|_{\xi=\xi_{2,m}}
\sim-(1/\alpha)^j\xi_{2,m}^{n-1}(\frac{m}{\xi_{2,m}})^{j-1}(\frac{\xi_{2,m}}{\xi_{2,m}-1})^{m+j}
\end{align*}
Last, taking into account the factor $\frac{1}{\binom{m+n-1}{m}}$, we have
\begin{align*}
&\frac{1}{\binom{m+n-1}{m}}\frac{d^{j-1}}{d\xi^{j-1}}\left[\frac{\xi^{n+m-1}}{(\xi-1)^{m+1}}\frac{(\xi-\xi_{1,m})^j}{[1-x_{l,m}^{(n)}\xi^{\alpha-l}(\xi-1)^l]^j}\right]\Bigg|_{\xi=\xi_{1,m}}\\
\sim&(-1)^{-m+(l-2)j-1}\frac{(1/\alpha)^j}{(j-1)!}\xi_{1,m}^{n+m}\frac{1}{m^{n-j}}\to0
\end{align*}
\begin{align*}
&\frac{1}{\binom{m+n-1}{m}}\frac{1}{(j-1)!}\frac{d^{j-1}}{d\xi^{j-1}}\left[\frac{\xi^{n+m-1}}{(\xi-1)^{m+1}}\frac{(\xi-\xi_{2,m})^j}{[1-y_{l,m}\xi^{\alpha-l}(\xi-1)^l]^j}\right]\Bigg|_{\xi=\xi_{2,m}}\\
\sim&-\frac{(1/\alpha)^j}{(j-1)!}(\frac{m}{\xi_{2,m}})^{j-n}\frac{1}{(1-1/\xi_{2,m})^{m}}\to\infty.
\end{align*}
When $l$ is odd the proof can be carried out similarly, one will have
\begin{align*}
&\frac{1}{\binom{m+n-1}{m}}\frac{1}{(j-1)!}\frac{d^{j-1}}{d\xi^{j-1}}\left[\frac{\xi^{n+m-1}}{(\xi-1)^{m+1}}\frac{(\xi-\xi_{m})^j}{[1-y_{l,m}\xi^{\alpha-l}(\xi-1)^l]^j}\right]\Bigg|_{\xi=\xi_{m}}\\
\sim&-\frac{(1/\alpha)^j}{(j-1)!}(\frac{m}{\xi_{m}})^{j-n}\frac{1}{(1-1/\xi_{m})^{m}}\to\infty.
\end{align*}
The proof is completed.

\section{Proof of Theorem \ref{mdp}}
\begin{proof}
By Theorem \ref{asymptotic}, $\forall \lambda>0$, one will have 
\begin{align*}
\log \phi_{n}\sim & ny_{n}^{1/\alpha}\sim \lambda^{1/\alpha}\beta_{n}^{1/(1-\alpha)}\\
\log \phi_{l,n}\sim& n\frac{1}{\eta_{n}}\sim nx_{l,n}^{1/\alpha}\sim[\frac{\alpha(1-\alpha)_{(l-1)}}{l!}\lambda]^{1/\alpha}\beta_{n}^{1/(1-\alpha)}\\
\log \phi_{m}^{(n)}\sim& m(y_{m}^{(n)})^{1/\alpha}\sim \lambda^{1/\alpha}\beta_{m}^{1/(1-\alpha)}\\
\log \phi_{l,m}^{(n)}\sim& m\frac{1}{\xi_{m}}\sim m(x_{l,m}^{(n)})^{1/\alpha}\sim[\frac{\alpha(1-\alpha)_{(l-1)}}{l!}\lambda]^{1/\alpha}\beta_{m}^{1/(1-\alpha)}
\end{align*}
Thus,$$
\psi_{\lambda}=\psi^{(n)}(\lambda)=\begin{cases}
\lambda^{1/\alpha} & \lambda>0\\
0&\lambda\leq0
\end{cases}
$$
and
$$
\psi_{l}(\lambda)=\psi_{l}^{(n)}(\lambda)=\begin{cases}
(\frac{\alpha(1-\alpha)_{(l-1)}}{l!}\lambda)^{1/\alpha} & \lambda>0\\
0&\lambda\leq0
\end{cases}
$$
Then by G\"{a}rtner-Ellis theorem, the proof is completed.
\end{proof}

\section{Appendix}
\textbf{[Proof of Lemma \ref{ss}]}:
\begin{proof}
First we consider the equation on the real line
$$
1-y_{l,n}x^{\alpha-l}(x-1)^l=0,\quad \alpha\in(0,1),x\in\mathbb{R},
$$
which is equivalent to
$
x^{\alpha-l}(x-1)^l=\frac{1}{y_{l,n}}.
$
Therefore, consider function
$
h(x)=x^{\alpha-l}(x-1)^l, x\in\mathbb{R}.
$
But there is definitely no solution on $\mathbb{R}^{-}$, for otherwise $x^{\alpha}$ is a pure complex number, which does not fit the equation. So we only need to consider $\mathbb{R}^{+}$. Moreover, whether $l$ is even or odd is crucial, for instance, when $l=4,\alpha=0.5$ and $l=3,\alpha=0.5$, we have the graph as in Figure \ref{el} and Figure \ref{el}.
\begin{figure}
  \centering
  \includegraphics[width=0.4\textwidth,height=3in]{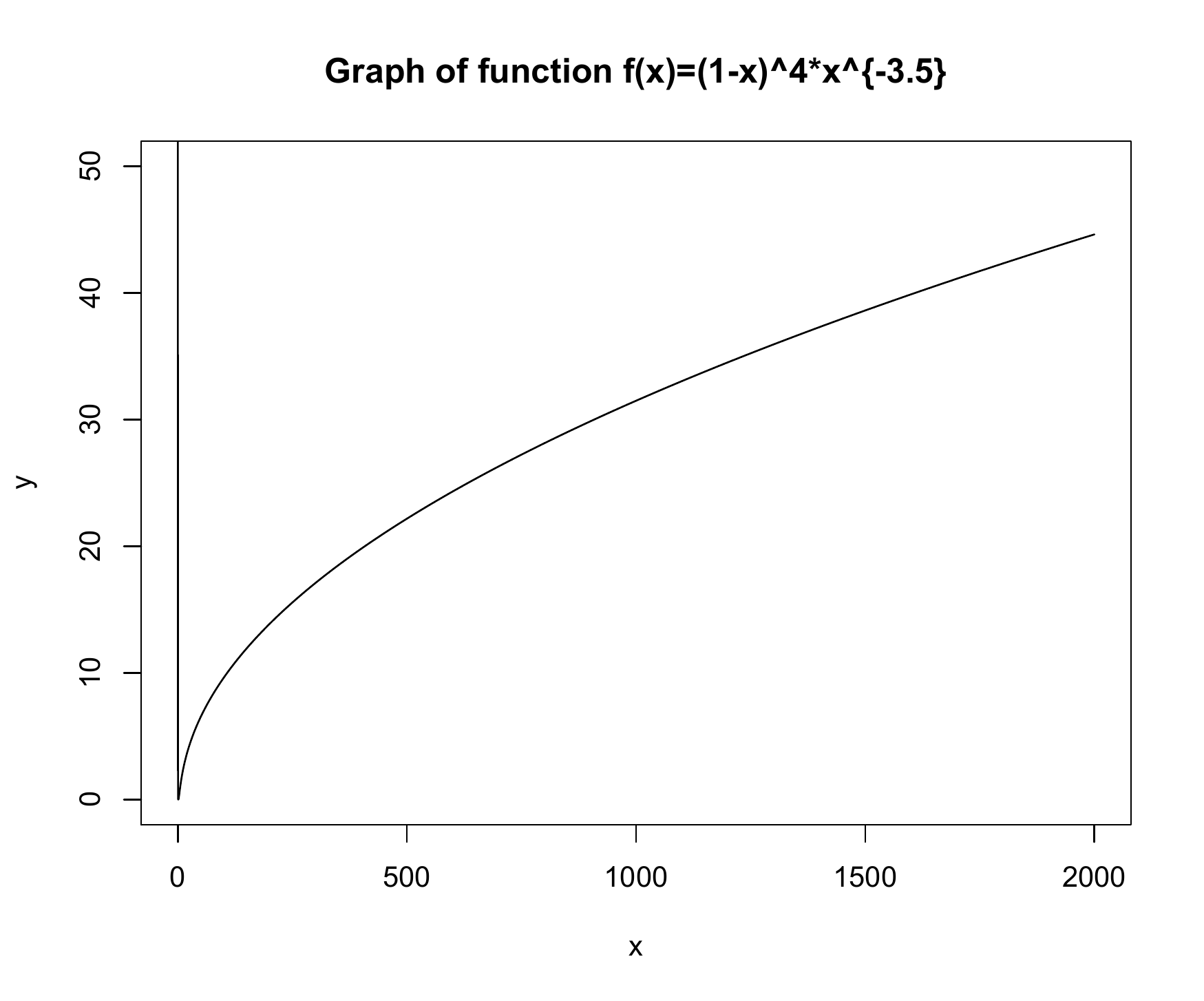}
    \includegraphics[width=0.4\textwidth,height=3in]{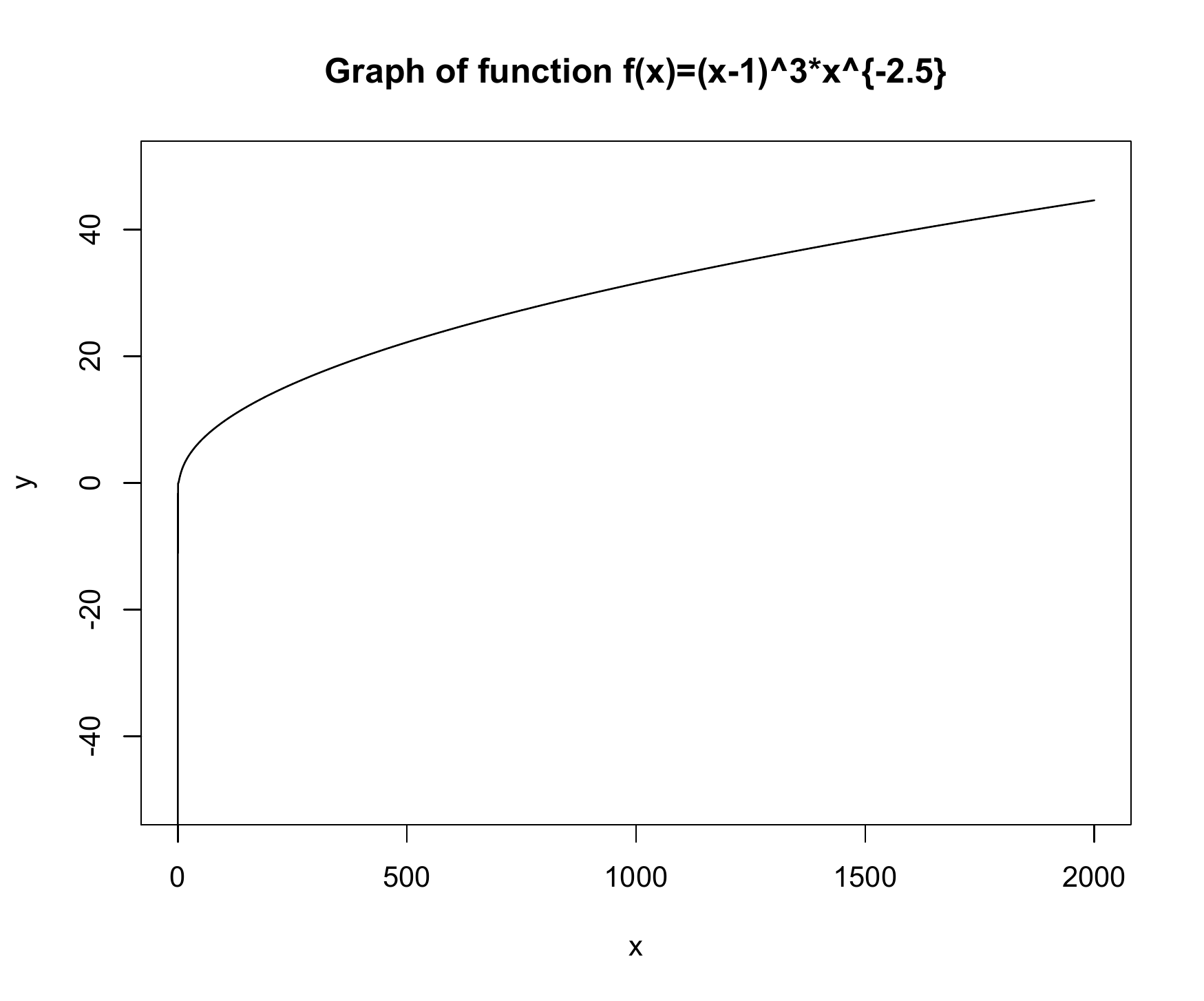}
   \caption{Function for even $l$ and odd $l$}
   \label{el}
\end{figure}

Let us find the derivative of $h(x)$.
$$
h^{'}(x)=\alpha x^{\alpha-l-1}(x-1)^{l-1}(x+\frac{l-\alpha}{\alpha})
$$
When $l$ is even, we know
$$
h^{'}(x)=\begin{cases}
>0&\mbox{ when } x>1\\
<0&\mbox{ when } 0<x<1
\end{cases}
$$
When $l$ is odd, we know $h^{'}(x)>0,x\in\mathbb{R}^{+}$. Therefore, when $l$ is even, there are always two real solutions $0<\xi_{1,n}<1$ and $\xi_{2,n}>1$. When $l$ is odd, there is only one real solution $\xi_{n}>1$. We also know that when $l$ is even,
$$
\xi_{1,n}^{\alpha-l}(\xi_{1,n}-1)^{l}=\frac{1}{y_{l,n}},\quad \xi_{2,n}^{\alpha-l}(\xi_{2,n}-1)^{l}=\frac{1}{y_{l,n}}.
$$
Since $1/y_{l,n}\to\infty$, then
$
\xi_{1,n}\sim y_{l,n}^{1/(l-\alpha)},\quad \xi_{2,n}\sim \frac{1}{y_{l,n}^{1/\alpha}}.
$
Similarly, when $l$ is odd, we have
$
\xi_{n}^{\alpha-l}(\xi_{n}-1)^{l}=\frac{1}{y_{l,n}}
$
and
$
\xi_{n}\sim \frac{1}{y_{l,n}^{1/\alpha}}.
$
Now we are going to show that the equation has only real solutions when $n$ is big enough. Since $\xi=0$ is not the solution. Then under the transformation $\eta=\frac{1}{\xi}$, the solution set of the following two equations has a one-to-one correspondence.
$$
1-y_{l,n}\xi^{\alpha-l}(\xi-1)^l=0\quad 1-y_{l,n}(1-\eta)^l\eta^{-\alpha}=0
$$
Thus, we only need to consider
$
1-y_{l,n}(1-\eta)^l\eta^{-\alpha}=0.
$
Rearranging the above equation, we have
$
\eta^{\alpha}=y_{l,n}(1-\eta)^l
$
If $\eta$ is a solution, then one can easily check that so is the conjugate $\bar{\eta}$. Therefore, we only need to consider the solutions on the upper half plane, where $\eta=Re^{i\theta}, 0\leq \theta<\pi$. Suppose that $\eta_{0,n}=R_{0,n}e^{i\theta_{0,n}},0\leq\theta_{0,n}<\pi$ is a solution. Then
$
|\eta_{0,n}|^\alpha=y_{l,n}|1-\eta_{0,n}|^l
$
and
$
R_{0,n}^\alpha=y_{l,n}(1+R_{0,n}^2-2R_{0,n}\cos\theta_{0,n})^{l/2}
$
Because $y_{l,n}\to0$ as $n\to\infty$. If $\{R_{0,n},n\geq1\}$ is bounded, then $R_{0,n}\to0$, and if $\{R_{0,n},n\geq1\}$ is unbounded, then $R_{0,n}\to\infty$. Then we divide the solution set into two parts:
\begin{align*}
S_{1}=&\{\eta_{0,n}\mid \mathrm{Re}(\eta_{0,n})\leq1,n\geq1\}\\
S_{2}=&\{\eta_{0,n}\mid \mathrm{Re}(\eta_{0,n})>1,n\geq1\}
\end{align*}
\begin{enumerate}
\item When $\eta_{0,n}\in S_{1}$, then $\mathrm{Re}(1-\eta_{0,n})\geq0$, and $\mathrm{Im}(1-\eta_{0,n})=\mathrm{Im}(-\eta_{0,n})<0$. Hence
    $$
\arg(1-\eta_{0,n})=\arctan\frac{\mathrm{Im}(1-\eta_{0,n})}{\mathrm{Re}(1-\eta_{0,n})}=\arctan\left(\frac{-R_{0,n}\sin\theta_{0,n}}{1-R_{0,n}\cos\theta_{0,n}}\right)
    $$
    Moreover, because $\eta_{0,n}^\alpha=y_{l,n}(1-\eta_{0,n})^l$, we know
    $$
    \alpha\arg{(\eta_{0,n})}=l\arg(1-\eta_{0,n})+2k_{n}\pi,\quad |k_{n}|\leq\frac{2l+1}{2}.
    $$
    Then
    $$
    \alpha\theta_{0,n}=l\arctan\left(\frac{-R_{0,n}\sin\theta_{0,n}}{1-R_{0,n}\cos\theta_{0,n}}\right)+2k_{n}\pi.
    $$
    For $\eta_{0,n}\in S_{1}$, then $\mathrm{Re}(\eta_{0,n})\leq1$, so $R_{0,n}\cos\theta_{0,n}\leq1$. If $R_{0,n}\to\infty$, then $\theta_{0,n}\to\frac{\pi}{2}$. Because $k_{n}$ is bounded, we can definitely pick a subsequence $k_{n_{t}}$ such that as $t\to\infty$
    $
    k_{n_{t}}\to k^{'}
    $
    and we also know
    $
    R_{0,n_{t}}\to\infty\quad \theta_{0,n_{t}}\to\frac{\pi}{2}.
    $
    Letting $t\to\infty$, we have
    $
    \alpha\frac{\pi}{2}=l\frac{\pi}{2}+2k^{'}\pi
    $
    Then
    $
    \alpha=l+4k^{'}
    $
    which is a contradiction for the left hand side is not an integer but the right hand side is integer.
    Thus, for $\eta_{0,n}\in S_{1}$ and $\eta_{0,n}=R_{0,n}e^{i\theta_{0,n}}$, we know $R_{0,n}\to0$. Because
    $$
    \frac{\sin\theta_{0,n}R_{0,n}}{1-R_{0,n}\cos\theta_{0,n}}\leq\frac{R_{0,n}}{1-R_{0,n}}\to0.
    $$
    then there is a positive integer $N_{0}$ such that for $n\geq N_{0}$ we have
    \begin{align*}
    |2k_{n}\pi|=&|\alpha\theta_{0,n}-l\arctan(\frac{\sin\theta_{0,n}R_{0,n}}{1-R_{0,n}\cos\theta_{0,n}})|\\
    \leq&|\alpha\theta_{0,n}|+|l\arctan(\frac{\sin\theta_{0,n}R_{0,n}}{1-R_{0,n}\cos\theta_{0,n}})|\\
   <&|\alpha\theta_{0,n}|+\frac{\pi}{2}<\pi+\frac{\pi}{2}=\frac{3\pi}{2}
    \end{align*}
    Then we know when $n\geq N_{0}$,  $k_{n}=0$. Then when $n\geq N_{0}$
    $$
    \alpha\theta_{0,n}=l\arctan\left(\frac{-R_{0,n}\sin\theta_{0,n}}{1-R_{0,n}\cos\theta_{0,n}}\right)\leq0
    $$
    But the left hand side $\alpha\theta_{0,n}\geq0$. Therefore, $\theta_{0,n}=0$. Hence, $S_{1}$ has only real number.
    \item When $\eta_{0,n}\in S_{2}$, we know $\mathrm{Re}(\eta_{0,n}>1)$, then $\mathrm{Re}(1-\eta_{0,n})<0$ and $\mathrm{Im}(1-\eta_{0,n})=-\mathrm{Im}(\eta_{0,n})<0$. Hence
        $$
        \arg(1-\eta_{0,n})=-\pi+\arctan(\frac{-R_{0,n}\sin\theta_{0,n}}{1-R_{0,n}\cos\theta_{0,n}})
        $$
    So
    $$
    \alpha \arg(\eta_{0,n})=l\arg(1-\eta_{0,n})+2k_{n}\pi, \quad |k_{n}|\leq\frac{2l+1}{2}
    $$
    and
    $$
    \alpha\theta_{0,n}=-l\pi+l\arctan(\frac{-R_{0,n}\sin\theta_{0,n}}{1-R_{0,n}\cos\theta_{0,n}})+2k_{n}\pi
    $$
    Since $1<\mathrm{Re}(\eta_{0,n})\leq|\eta_{0,n}|=R_{0,n}$, then $R_{0,n}\to\infty$ as $n\to\infty$. So
    \begin{align*}
   &\alpha\theta_{0,n}=-l\pi+l\arctan\frac{\sin\theta_{0,n}}{\cos\theta_{0,n}-\frac{1}{R_{0,n}}}+2k_{n}\pi\\
    =&-l\pi+l\arctan(\tan\theta_{0,n})+2k_{n}\pi+l(\arctan\frac{\sin\theta_{0,n}}{\cos\theta_{0,n}-\frac{1}{R_{0,n}}}-\arctan\tan\theta_{0,n})\\
    =&-l\pi+l\theta_{0,n}+2k_{n}\pi+l(\arctan\frac{\sin\theta_{0,n}}{\cos\theta_{0,n}-\frac{1}{R_{0,n}}}-\arctan\tan\theta_{0,n})\\
    =&-l\pi+l\theta_{0,n}+l\arctan(\frac{\sin\theta_{0,n}}{1-\cos\theta_{0,n}/R_{0,n}}\frac{1}{R_{0,n}})
    \end{align*}
    where we have use the formula
    $$
    \arctan\alpha-\arctan\beta=\arctan(\frac{\alpha-\beta}{1+\alpha\beta})
    $$
    There is a positive integer $N_{0}$ such that for $n\geq N_{0}$
    $$
    |\frac{\sin\theta_{0,n}}{1-\cos\theta_{0,n}/R_{0,n}}|\leq|\frac{1}{1-\frac{1}{R_{0,n}}}|<2.
    $$
    Notice that $0<\theta_{0,n}<\frac{\pi}{2}$ and $|k_{n}|\leq\frac{2l+1}{2}$.There is a sequence $n_{t}$ such that as $t\to\infty$
    $
    k_{n_{t}}\to k^{'},\quad \theta_{0,n_{t}}\to \theta.
    $
    We claim that $\theta=0$ otherwise $\theta>0$. We have
    \begin{align*}
    \alpha\theta=&-l\pi+l\theta+2k^{'}\pi\\
    (\alpha-l)\theta=&(2k^{'}-l)\pi\\
    (\alpha-l)\frac{\theta}{\pi}=&(2k^{'}-l)
    \end{align*}
    We know $0<\frac{\theta}{\pi}<\frac{1}{2}$. We can choose a rational number $\frac{p}{l}$ such that
    $$
    0<\frac{\theta}{\pi}-\frac{p}{l}<\frac{1}{l}.
    $$
    one possible way to get the above rational numbers is to divide $[0,1]$ into $l$ equal parts. $\frac{\theta}{\pi}$ will fall in one of the small interval. We choose the left end point of the interval,denoted as $\frac{p}{l}$, which is what we need. We can thus show that the left hand side of $(\alpha-l)\frac{\theta}{\pi}=(2k^{'}-l)$ is not integer but the right hand side is integer, which is a contradiction. Therefore, the claim is shown. Now we show that
    $
     (\alpha-l)\frac{\theta}{\pi}
    $
    is not an integer. To this end, let $r_{l}=\frac{\theta}{\pi}-\frac{p}{l}$ we have $0<r_{l}<\frac{1}{l}$ and
    $$
    (\alpha-l)\frac{\theta}{\pi}=\alpha\frac{\theta}{\pi}-l\frac{\theta}{\pi}
    =\alpha\frac{\theta}{\pi}-l(\frac{p}{l}+r_{l})=\alpha\frac{\theta}{\pi}-lr_{l}-p
    $$
    Since $0<lr_{l}<1$ and $0<\alpha\frac{\theta}{\pi}<1$, we know $\alpha\frac{\theta}{\pi}-lr_{l}$ is not an integer, but $p$ is an integer. So $(\alpha-l)\frac{\theta}{\pi}$ is not an integer. Our claim is thus proved.

    Therefore, $\theta_{0,n}\to0$ and
    $
    2k^{'}=l.
    $
    Thus all the subsequence of $k_{n}$ converges to $\frac{l}{2}$. Therefore, because $k_{n}$ is an integer sequence, we have
    $
    k_{n}=\frac{l}{2}\quad \forall n\geq N_{0}.
    $
    So
    $
    \alpha\theta_{0,n}=-l\pi+l\theta_{0,n}+2k_{n}\pi=-l\pi+l\theta_{0,n}+l\pi=l\theta_{0,n}.
    $
    Then $\alpha\theta_{0,n}=l\theta_{0,n}$, then $\theta_{0,n}=0\quad \forall n\geq N_{0}.$ Thus $S_{2}$ has only real numbers.
\end{enumerate}
\end{proof}
\textbf{Proof of Lemma \ref{residue}}:
\begin{proof}
We apply Leibnitz's formula to equation
$F(\xi)G(\xi)=1$, then
$$
0=\sum_{p=0}^{k}\binom{k}{p}F^{(p)}(\xi)G^{(k-p)}(\xi).
$$
So we have a recursive relation
$$
G^{(k)}(\xi)=-\sum_{p=1}^{k}\binom{k}{p}\frac{F^{(p)}(\xi)}{F(\xi)}G^{(k-p)}(\xi).
$$
Then
$$
G^{(k)}(0)=-\sum_{p=1}^{k}\frac{k!j!}{(p+j)!}\frac{A_{p+j}}{A_{j}}\frac{1}{(k-p)!}G^{(k-p)}(0).
$$
We will use mathematical induction to show this formula. When $k=1$,
$$
G^{'}(0)=-\frac{F^{'}(0)}{F(0)}G(0)=-\frac{j!}{(j+1)!}\frac{A_{j+1}}{A_{j}}\frac{j!}{A_{j}}
$$
the formula is true. Now we assume that 
$$
G^{(k)}(0)=\left[\sum_{s=1}^k(-1)^s\sum_{\eta\in\mathcal{P}_{s}(k)}\frac{(j!)^sk!}{(j+\eta_{1})!\cdots(j+\eta_{s})!}\mathcal{A}_{s}^j(\eta)\right]\frac{j!}{A_{j}},
$$
Then
\begin{align*}
&G^{(k+1)}(0)=-\sum_{p=1}^{k}\frac{(k+1)!j!}{(p+j)!}\frac{A_{p+j}}{A_{j}}\frac{1}{(k+1-p)!}G^{(k+1-p)}(0)
-\frac{(k+1)!j!}{(k+1+j)!}\frac{A_{k+1+j}}{A_{j}}\frac{j!}{A_{j}}\\
=&-\sum_{p=1}^{k}\frac{(k+1)!j!}{(p+j)!}\frac{A_{p+j}}{A_{j}}\frac{1}{(k+1-p)!}G^{(k+1-p)}(0)-\sum_{\eta\in\mathcal{P}_{1}(k+1)}\frac{j!(k+1)!}{(j+\eta_{1})!}\mathcal{A}_{1}^j(\eta)\frac{j!}{A_{j}}
\end{align*}
Replacing  $G^{(k+1-p)}(0)$ by
$$
G^{(k+1-p)}(0)=\left[\sum_{s=1}^{k+1-p}(-1)^s\sum_{\eta\in\mathcal{P}_{s}(k+1-p)}\frac{(j!)^s(k+1-p)!}{(j+\eta_{1})!\cdots(j+\eta_{s})!}\mathcal{A}_{s}^j(\eta)\right]\frac{j!}{A_{j}},
$$
we have
\begin{align*}
&G^{(k+1)}(0)=-\sum_{\eta\in\mathcal{P}_{1}(k+1)}\frac{j!(k+1)!}{(j+\eta_{1})}\mathcal{A}_{1}^j(\eta)\frac{j!}{A_{j}}-\sum_{p=1}^{k}\frac{(k+1)!j!}{(p+j)!}\frac{A_{p+j}}{A_{j}}\frac{1}{(k+1-p)!}\\
&\left[\sum_{s=1}^{k+1-p}(-1)^s\sum_{\eta\in\mathcal{P}_{s}(k+1-p)}\frac{(j!)^s(k+1-p)!}{(j+\eta_{1})!\cdots(j+\eta_{s})!}\mathcal{A}_{s}^j(\eta)\right]\frac{j!}{A_{j}}\\
=&\left[\sum_{p=1}^{k}\sum_{s=1}^{k+1-p}(-1)^{s+1}\sum_{\eta\in\mathcal{P}_{s}(k+1-p)}\frac{j!^{s+1}(k+1)!}{(j+p)!(j+\eta_{1})!\cdots(j+\eta_{s})!}\frac{A_{j+p}}{A_{j}}\mathcal{A}_{s}^j\right]\frac{j!}{A_{j}}\\
&-\sum_{\eta\in\mathcal{P}_{1}(k+1)}\frac{j!(k+1)!}{(j+\eta_{1})}\mathcal{A}_{1}^j(\eta)\frac{j!}{A_{j}}\\
=&\left[\sum_{s=2}^{k+1}(-1)^{s}\sum_{\eta\in\mathcal{P}_{s}(k+1)}\frac{j!^s(k+1)!}{(j+\eta_{1})!\cdots(j+\eta_{s})!}\mathcal{A}_{s}^j\right]\frac{j!}{A_{j}}-\sum_{\eta\in\mathcal{P}_{1}(k+1)}\frac{j!(k+1)!}{(j+\eta_{1})}\mathcal{A}_{1}^j(\eta)\frac{j!}{A_{j}}\\
=&\left[\sum_{s=1}^{k+1}(-1)^{s}\sum_{\eta\in\mathcal{P}_{s}(k+1)}\frac{j!^s(k+1)!}{(j+\eta_{1})!\cdots(j+\eta_{s})!}\mathcal{A}_{s}^j\right]\frac{j!}{A_{j}}.
\end{align*}
The proof is completed.
\end{proof}

\bibliography{ref}{}

\begin{thebibliography}{1}

\bibitem{MR3167885}
{\sc S.~Favaro and S.~Feng}, {\em Asymptotics for the number of blocks in a
  conditional {E}wens-{P}itman sampling model}, Electron. J. Probab., 19
  (2014), pp.~no. 21, 15.

\bibitem{MR3335831}
\leavevmode\vrule height 2pt depth -1.6pt width 23pt, {\em Large deviation
  principles for the {E}wens-{P}itman sampling model}, Electron. J. Probab., 20
  (2015), pp.~no. 40, 26.

\bibitem{MR3850070}
{\sc S.~Favaro, S.~Feng, and F.~Gao}, {\em Moderate {D}eviations for
  {E}wens-{P}itman {S}ampling {M}odels}, Sankhya A, 80 (2018), pp.~330--341.

\bibitem{MR1661315}
{\sc S.~Feng and F.~M. Hoppe}, {\em Large deviation principles for some random
  combinatorial structures in population genetics and {B}rownian motion}, Ann.
  Appl. Probab., 8 (1998), pp.~975--994.

\bibitem{MR2160320}
{\sc A.~Gnedin and J.~Pitman}, {\em Exchangeable {G}ibbs partitions and
  {S}tirling triangles}, Zap. Nauchn. Sem. S.-Peterburg. Otdel. Mat. Inst.
  Steklov. (POMI), 325 (2005), pp.~83--102, 244--245.

\bibitem{MR509954}
{\sc J.~F.~C. Kingman}, {\em The representation of partition structures}, J.
  London Math. Soc. (2), 18 (1978), pp.~374--380.

\bibitem{MR2434179}
{\sc A.~Lijoi, I.~Pr\"{u}nster, and S.~G. Walker}, {\em Bayesian nonparametric
  estimators derived from conditional {G}ibbs structures}, Ann. Appl. Probab.,
  18 (2008), pp.~1519--1547.

\bibitem{MR1337249}
{\sc J.~Pitman}, {\em Exchangeable and partially exchangeable random
  partitions}, Probab. Theory Related Fields, 102 (1995), pp.~145--158.

\bibitem{MR2245368}
{\sc J.~Pitman}, {\em Combinatorial stochastic processes}, vol.~1875 of Lecture
  Notes in Mathematics, Springer-Verlag, Berlin, 2006.

\end{thebibliography}
\bibliographystyle{siam}

\end{document}